\documentclass{amsproc}

\usepackage{hyperref}
\usepackage{mathrsfs}
\usepackage{multirow}
\usepackage{array}

\newtheorem{theorem}{Theorem}[section]
\newtheorem{proposition}[theorem]{Proposition}
\newtheorem{lemma}[theorem]{Lemma}
\newtheorem{corollary}[theorem]{Corollary}

\theoremstyle{definition}
\newtheorem{definition}[theorem]{Definition}
\newtheorem{example}[theorem]{Example}

\theoremstyle{remark}
\newtheorem{remark}[theorem]{Remark}

\numberwithin{equation}{section}

\newcommand{\abs}[1]{\lvert#1\rvert}

\newcommand{\p}{\partial}
\newcommand{\dslash}{/\mskip-6mu/}
\newcommand{\one}{{{\mathchoice {\rm 1\mskip-4mu l} {\rm 1\mskip-4mu l}
{\rm 1\mskip-4.5mu l} {\rm 1\mskip-5mu l}}}}

\newcommand{\C}{{\mathbb{C}}}

\newcommand{\R}{{\mathbb{R}}}
\newcommand{\T}{{\mathbb{T}}}

\newcommand{\cE}{{\mathcal E}}
\newcommand{\cF}{{\mathcal F}}
 
\newcommand{\cH}{{\mathcal H}}
\newcommand{\cI}{{\mathcal I}}

\newcommand{\cK}{{\mathcal K}}
\newcommand{\cL}{{\mathcal L}} 
\newcommand{\cM}{{\mathcal M}}

\newcommand{\sC}{\mathscr{C}}

\newcommand{\sG}{\mathscr{G}}    
\newcommand{\sH}{\mathscr{H}}    
    
\newcommand{\sJ}{\mathscr{J}}    
\newcommand{\sK}{\mathscr{K}}    
\newcommand{\sL}{\mathscr{L}}

\newcommand{\sP}{\mathscr{P}}

\newcommand{\sS}{\mathscr{S}}    
\newcommand{\sT}{\mathscr{T}}    
    
\newcommand{\sV}{\mathscr{V}}    
\newcommand{\sW}{\mathscr{W}}    
    
\newcommand{\sY}{\mathscr{Y}}    
\newcommand{\sZ}{\mathscr{Z}}

\newcommand{\bB}{\mathbf{B}}   
\newcommand{\bL}{\mathbf{L}}  

\newcommand{\sn}{{\mathsf{n}}}

\newcommand{\ps}{{\mathsf{ps}}}

\newcommand{\Om}{{\Omega}}
\newcommand{\om}{{\omega}}

\newcommand{\ghat}{{\widehat{g}}}
\newcommand{\hhat}{{\widehat{h}}}

\newcommand{\xhat}{{\widehat{x}}}

\newcommand{\Fhat}{{\widehat{F}}}
\newcommand{\Ghat}{{\widehat{G}}}

\newcommand{\Jhat}{{\widehat{J}}}

\newcommand{\Shat}{{\widehat{S}}}

\newcommand{\lambdahat}{{\widehat{\lambda}}}

\newcommand{\omhat}{{\widehat{\omega}}}

\newcommand{\Thetahat}{{\widehat{\Theta}}}
\newcommand{\Richat}{{\widehat{\Ric}}}

\newcommand{\im}{{\mathrm{im}}}

\newcommand{\trace}{{\mathrm{trace}}} 
 
\newcommand{\id}{{\mathrm{id}}}

\newcommand{\INT}{{\mathrm{int}}}

\newcommand{\grad}{{\mathrm{grad}}}

\renewcommand{\Im}{{\mathrm{Im}}}  
\newcommand{\Lie}{{\mathrm{Lie}}} 
\newcommand{\Aut}{{\mathrm{Aut}}}   
 
\newcommand{\Diff}{{\mathrm{Diff}}} 
\newcommand{\Vect}{{\mathrm{Vect}}}  
\newcommand{\Symp}{{\mathrm{Symp}}}   
\newcommand{\Ham}{{\mathrm{Ham}}}

\newcommand{\Flux}{{\mathrm{Flux}}}

\newcommand{\SE}{{\mathrm{SE}}}

\newcommand{\bi}{{\mathbf{i}}}

\newcommand{\ex}{{\mathrm{ex}}}
\renewcommand{\csc}{{\mathrm{csc}}}
\newcommand{\cscK}{{\mathrm{cscK}}}

\newcommand{\WP}{{\mathrm{WP}}}

\newcommand{\rG}{{\mathrm{G}}}

\newcommand{\SU}{{\mathrm{SU}}} 
 
\newcommand{\SL}{{\mathrm{SL}}}

\newcommand{\fg}{{\mathfrak{g}}} 

\newcommand{\Cinf}{C^{\infty}}

\newcommand{\CP}{{\mathbb{C}\mathrm{P}}}
\newcommand{\KE}{{\mathrm{KE}}}

\newcommand{\Ric}{{\mathrm{Ric}}}

\newcommand{\inner}[2]{\langle #1, #2\rangle}
\newcommand{\INNER}[2]{\left\langle #1, #2\right\rangle}

\newlength{\dtildeheight}

\newlength{\dhatheight}

\def\NABLA#1{{\mathop{\nabla\kern-.5ex\lower1ex\hbox{$#1$}}}}
\def\Nabla#1{{\nabla\kern-.5ex{}_{#1}}}
\def\Tabla#1{{\widetilde\nabla\kern-.5ex{}_{#1}}}
\def\DTabla#1{{{\widetilde{\vphantom{\rule{1pt}{9.6pt}}\smash{\widetilde{\nabla}}}}\kern-.5ex{}_{#1}}}
\def\SDTabla#1{{{\widetilde{\vphantom{\rule{1pt}{6.8pt}}\smash{\widetilde{\nabla}}}}\kern-.5ex{}_{#1}}}
\def\Habla#1{{\widehat\nabla\kern-.5ex{}_{#1}}}
\def\DHabla#1{{{\widehat{\vphantom{\rule{1pt}{9.6pt}}\smash{\widehat{\nabla}}}}\kern-.5ex{}_{#1}}}
\def\SDHabla#1{{{\widehat{\vphantom{\rule{1pt}{6.8pt}}\smash{\widehat{\nabla}}}}\kern-.5ex{}_{#1}}}
\def\abs#1{\mathopen|#1\mathclose|}
\def\Abs#1{\left|#1\right|}
\def\norm#1{\mathopen\|#1\mathclose\|}
\def\Norm#1{\left\|#1\right\|}
\renewcommand{\p}{{\partial}}

\begin{document}

\title[Ricci form and Teichm\"uller spaces]{A moment map interpretation of the Ricci form, \\ 
K\"ahler--Einstein structures, and Teichm\"uller spaces}

\author{Oscar Garc\'{\i}a-Prada}
\address{ICMAT Madrid}
\email{oscar.garcia-prada@icmat.es}

\author{Dietmar~Salamon}
\address{ETH Z\"urich}
\email{salamond@math.ethz.ch}

\subjclass{53D20, 53Q20, 53Q25, 14J10}  
\date{18 April 2020}

\dedicatory{This paper is dedicated to the memory of Boris Dubrovin.}

\keywords{moment map, Ricci form, K\"ahler--Einstein, Teichm\"uller space}

\begin{abstract}
This paper surveys the role of moment maps in K\"ahler geometry.
The first section discusses the Ricci form as a moment map 
and then moves on to moment map interpretations
of the K\"ahler--Einstein condition 
and the scalar curvature (Quillen--Fujiki--Donaldson).
The second section examines the ramifications of these results
for various Teich\-m\"uller spaces and their Weil--Petersson
symplectic forms and explains how these arise naturally from
the construction of symplectic quotients.  The third section
discusses a symplectic form introduced by Donaldson on the
space of Fano complex structures.
\end{abstract}

\maketitle


\vspace{-1pt}

\section{The Ricci form}\label{sec:RICCI}

This section explains how the Ricci form appears as a moment map
for the action of the group of exact volume preserving diffeomorphisms 
on the space of almost complex structures.   A direct consequence of this 
observation is the Quillen--Fujiki--Donaldson Theorem about the scalar curvature
as a moment map for the action of the group of Hamiltonian symplectomorphisms
on the space of compatible almost complex structures on a symplectic manifold.
This section also discusses how the K\"ahler--Einstein condition
can be interpreted as a moment map equation.


\subsection{The Ricci form as a moment map}\label{subsec:RICCI}

Let~$M$ be a closed oriented $2\sn$-manifold equipped 
with a positive volume form~${\rho\in\Om^{2\sn}(M)}$.   
Then the space~$\sJ(M)$ of all almost complex structures 
on~$M$ that are compatible with the orientation can be thought of 
as an infinite-dimensional symplectic manifold. 
Its tangent space at~${J\in\sJ(M)}$ is  the space of all complex anti-linear 
endomorphisms~${\Jhat:TM\to TM}$ of the tangent bundle
(see~\cite[Section~2]{GPST}) and thus can be identified 
with the space~$\Om^{0,1}_J(M,TM)$ of complex 
anti-linear $1$-forms on~$M$ with values in the tangent bundle.  
The symplectic form~$\Om_\rho$ is given by
\begin{equation}\label{eq:OM}
\Om_{\rho,J}(\Jhat_1,\Jhat_2) 
:= \tfrac{1}{2}\int_M\trace\Bigl(\Jhat_1J\Jhat_2\Bigr)\rho
\end{equation}
for~${J\in\sJ(M)}$ and~${\Jhat_1,\Jhat_2\in T_J\sJ(M)=\Om^{0,1}_J(M,TM)}$.

The symplectic form is preserved by the action of the
group~$\Diff(M,\rho)$ of volume preserving diffeomorphisms.  
Denote the identity component by~${\Diff_0(M,\rho)}$ and the subgroup
of exact volume preserving diffeomorphisms 
(that are isotopic to the identity via an isotopy that is generated 
by a smooth family of exact divergence-free vector fields) 
by~${\Diff^\ex(M,\rho)}$.

Consider the submanifold~${\sJ_0(M)\subset\sJ(M)}$ of all 
almost complex structures that are compatible with the orientation 
and have real first Chern class zero.
It was shown in~\cite{GPST} that the action 
of~$\Diff^\ex(M,\rho)$ on~${\sJ_0(M)}$ is Hamiltonian 
and that twice the Ricci form appears as a moment map.  
To make this precise, note that the Lie algebra of~$\Diff^\ex(M,\rho)$ 
is the space of exact divergence-free vector fields and can be  
identified with the quotient space~${\Om^{2\sn-2}(M)/\ker d}$ via the 
correspondence~${\Om^{2\sn-2}(M)\to\Vect^\ex(M,\rho):\alpha\mapsto Y_\alpha}$,
defined by
$$
\iota(Y_\alpha)\rho = d\alpha.
$$
The dual space of the quotient~${\Om^{2\sn-2}(M)/\ker d}$
can formally be thought of as the space of exact $2$-forms, 
in that every exact $2$-form~${\tau\in d\Om^1(M)}$ gives rise to a continuous linear 
functional~${\Om^{2\sn-2}(M)/\ker d\to\R:[\alpha]\mapsto\int_M\tau\wedge\alpha}$.

The {\bf Ricci form}~${\Ric_{\rho,J}\in\Om^2(M)}$ associated to a 
volume form~$\rho$ and an almost complex structure~$J$,
both inducing the same orientation of~$M$, is defined by
\begin{equation}\label{eq:RICCI}
\begin{split}
\Ric_{\rho,J}(u,v) 
:= \tfrac{1}{2}\trace\bigl(JR^\nabla(u,v)\bigr) 
+ \tfrac{1}{4}\trace\bigl((\Nabla{u}J)J(\Nabla{v}J)\bigr) 
+ \tfrac{1}{2}d\lambda^\nabla_J(u,v)
\end{split}
\end{equation}
for~${u,v\in\Vect(M)}$.  Here~$\nabla$ is a torsion-free connection 
on~$TM$ that preserves the volume form~$\rho$
and the $1$-form~${\lambda^\nabla_J\in\Om^1(M)}$ is defined by
$$
\lambda^\nabla_J(u):=\trace\bigl((\nabla J)u\bigr)
$$
for~${u\in\Vect(M)}$.   The Ricci form is independent of the choice
of the torsion-free $\rho$-connection used to define it and
it is closed and represents the cohomology class~$2\pi c_1^\R(J)$.
Its dependence on the volume form is governed by the identity
\begin{equation}\label{eq:RICf}
\Ric_{e^f\rho,J}=\Ric_{\rho,J}+\tfrac{1}{2}d(df\circ J).
\end{equation}
for~${J\in\sJ(M)}$ and~${f\in\Om^0(M)}$,
and the map~$(\rho,J)\mapsto\Ric_{\rho,J}$ is equivariant 
under the action of the diffeomorphism group, i.e.\
\begin{equation}\label{eq:RICphi}
\Ric_{\phi^*\rho,\phi^*J}=\phi^*\Ric_{\rho,J}
\end{equation}
for all~${J\in\sJ(M)}$ and all~${\phi\in\Diff(M)}$.

The definition of the Ricci form in~\eqref{eq:RICCI} arises 
as a special case of a general moment map identity in~\cite{DON3} 
for sections of certain $\SL(2\sn,\R)$ fiber bundles.  
If~$\rho$ is the volume form of a K\"ahler metric and~$\nabla$
is the Levi-Civita connection, then~${\nabla J=0}$ and
hence the last two terms in~\eqref{eq:RICCI} vanish 
and~$\Ric_{\rho,J}$ is the standard Ricci form.  In general, 
the second summand in~\eqref{eq:RICCI} is a correction term 
which gives rise to a closed $2$-form that represents $2\pi$
times the first Chern class, and the last summand is a further 
correction term that makes the Ricci form independent of the choice
of the torsion-free $\rho$-connection~$\nabla$.   If~$J$ is compatible 
with a symplectic form~$\om$ and~$\nabla$ is the Levi-Civita 
connection of the Riemannian metric~$\om(\cdot,J\cdot)$,
then~${\lambda^\nabla_J=0}$.  In the integrable case the 
$2$-form~$\bi\Ric_{\rho,J}$ is the curvature of the Chern connection 
on the canonical bundle associated to the Hermitian structure 
determined by~$\rho$ and hence is a~$(1,1)$-form.

\begin{theorem}[\cite{GPST}]\label{thm:RICCI}
The action of the group~$\Diff^\ex(M,\rho)$ on the space~$\sJ_0(M)$ 
with the symplectic form~\eqref{eq:OM} is a Hamiltonian group action
and is generated by the~$\Diff(M,\rho)$-equivariant
moment map~${\sJ_0(M)\to d\Om^1(M):J\mapsto2\Ric_{\rho,J}}$, i.e.\ 
\begin{equation}\label{eq:RICCIMOMENT}
\int_M2\Richat_\rho(J,\Jhat)\wedge\alpha = \Om_{\rho,J}(\Jhat,\cL_{Y_\alpha}J)
\end{equation}
for all~${J\in\sJ_0(M)}$, all~${\Jhat\in\Om^{0,1}_J(M,TM)}$
and all~${\alpha\in\Om^{2\sn-2}(M)}$, where
$$
\Richat_\rho(J,\Jhat) := \left.\tfrac{d}{dt}\right|_{t=0}\Ric_{\rho,J_t}
$$
for any smooth path~${\R\to\sJ(M):t\mapsto J_t}$ 
satisfying~${J_0=J}$ and~${\left.\tfrac{d}{dt}\right|_{t=0}J_t=\Jhat}$. 
\end{theorem}

Theorem~\ref{thm:RICCI} is based on ideas in~\cite{DON3}.
We emphasize that equation~\eqref{eq:RICCIMOMENT} does not 
require the vanishing of the first Chern class. Its proof in~\cite{GPST}
relies on the construction of a $1$-form~$\Lambda_\rho$ 
on~$\sJ(M)$ with values in the space of $1$-forms on~$M$. 
For~${J\in\sJ(M)}$ and~${\Jhat\in\Om^{0,1}_J(M,TM)}$ the 
$1$-form~${\Lambda_\rho(J,\Jhat)\in\Om^1(M)}$ is defined by 
\begin{equation}\label{eq:LAMBDA}
\bigl(\Lambda_\rho(J,\Jhat)\bigr)(u)
:= \trace\bigl((\nabla\Jhat)u+\tfrac12\Jhat J\Nabla{u}J\bigr)
\end{equation}
for~${u\in\Vect(M)}$, where~$\nabla$ is a torsion-free $\rho$-connection.
As before, ${\Lambda_\rho(J,\Jhat)}$ is independent of the 
choice of~$\nabla$.  Moreover, $\Lambda_\rho$ satisfies 
the following identities.  

\begin{proposition}[\cite{GPST}]\label{prop:LAMBDA}
Let~${J\in\sJ(M)}$, ${\Jhat\in\Om^{0,1}_J(M,TM)}$, and~${v\in\Vect(M)}$.  
Denote the divergence of~$v$ by~${f_v:=d\iota(v)\rho/\rho}$.  Then
\begin{eqnarray}
d\bigl(\Lambda_\rho(J,\Jhat)\bigr)
&\!\!\!\!=\!\!\!\!& 
2\Richat_\rho(J,\Jhat),
\label{eq:LAMBDA1}\\  
\int_M\Lambda_\rho(J,\Jhat)\wedge\iota(v)\rho
&\!\!\!\!=\!\!\!\!& 
\tfrac{1}{2}\int_M\trace\bigl(\Jhat J\cL_vJ\bigr)\rho,
\label{eq:LAMBDA2}\\
\Lambda_\rho(J,\cL_vJ)
&\!\!\!\!=\!\!\!\!& 
2\iota(v)\Ric_{\rho,J} - df_v\circ J + df_{Jv}.
\label{eq:LAMBDA3}
\end{eqnarray}
\end{proposition}

For a proof of Proposition~\ref{prop:LAMBDA} see~\cite[Theorems~2.6 \& 2.7]{GPST},
and note that equation~\eqref{eq:RICCIMOMENT} in Theorem~\ref{thm:RICCI} 
follows directly from~\eqref{eq:LAMBDA1} and~\eqref{eq:LAMBDA2} 
with~${v=Y_\alpha}$.

\begin{remark}\label{rmk:LAMBDA}
Two useful equations (see~\cite[Lemma~2.12]{GPST}) are
\begin{equation}\label{eq:LAMBDA4}
\cL_XJ=2J\bar\p_JX,\qquad
\Lambda_\rho(J,\Jhat) = \iota(2J\bar\p_J^*\Jhat^*)\om
\end{equation}
for~${J\in\sJ(M)}$, ${\Jhat\in\Om^{0,1}_J(M,TM)}$, ${X\in\Vect(M)}$.
Here~$\om$ is a nondegenerate $2$-form on~$M$ such that~${\om^\sn/\sn! = \rho}$
and~${\inner{\cdot}{\cdot}:=\om(\cdot,J\cdot)}$ is a Riemannian metric.
\end{remark}

For any Hamiltonian group action the zero set of the moment map 
is invariant under the group action and its orbit space is called 
the Marsden--Weinstein quotient.  In the case at hand this quotient 
is the space of exact volume preserving isotopy classes 
of Ricci-flat almost complex structures given by
\begin{equation}\label{eq:WZERO}
\begin{split}
\sW_0(M,\rho) 
&:= 
\sJ_0(M,\rho)/\Diff^\ex(M,\rho),\\
\sJ_0(M,\rho) 
&:= 
\left\{J\in\sJ_0(M)\,|\,\Ric_{\rho,J}=0\right\}.
\end{split}
\end{equation}
In the finite-dimensional setting it follows directly from the definitions
that an element of the zero set of the moment map is a regular point 
for the moment map (i.e.\ its derivative is surjective) if and only if
the isotropy subgroup is discrete. It was shown in~\cite[Theorem~2.11]{GPST}
that this carries over to the present situation.

\begin{proposition}[\cite{GPST}]\label{prop:REGULAR}
Fix an element~${J\in\sJ(M)}$. 

\smallskip\noindent{\bf (i)}
Let~${\lambdahat\in\Om^1(M)}$. Then there exists a~${\Jhat\in\Om^{0,1}_J(M,TM)}$ 
such that~${\Richat_\rho(J,\Jhat)=d\lambdahat}$ if and only if~${\int_M d\lambdahat\wedge\alpha=0}$
for all~${\alpha\in\Om^{2\sn-2}(M)}$ with~${\cL_{Y_\alpha}J=0}$.

\smallskip\noindent{\bf (ii)}
Let~${\Jhat\in\Om^{0,1}_J(M,TM)}$. Then there exists a~${\alpha\in\Om^{2\sn-2}(M)}$ 
such that~${\cL_{Y_\alpha}J=\Jhat}$ if and only if~${\Om_{\rho,J}(\Jhat,\Jhat')=0}$ for 
all~${\Jhat'\in\Om^{0,1}_J(M,TM)}$ with~${\Richat_\rho(J,\Jhat')=0}$.
\end{proposition}

Call an almost complex structure~${J\in\sJ(M)}$ {\bf regular} if 
there is no nonzero exact divergence-free $J$-holomorphic vector field.
The set of regular almost complex structures is open and 
the next proposition gives a regularity criterion.
It shows that every K\"ahlerable complex structure 
with real first Chern class zero is regular.

\begin{proposition}[\cite{GPST}]\label{prop:ISOTROPY}
Assume that~${J\in\sJ_0(M)}$ satisfies~${\Ric_{\rho,J}=0}$
and is compatible with a 
symplectic form~${\om\in\Om^2(M)}$ such that~${\om^\sn/\sn!=\rho}$
and the homomorphism~$H^1(M;\R)\to H^{2\sn-1}(M,\R):
[\lambda]\mapsto[\lambda\wedge\om^{\sn-1}/(\sn-1)!]$
is bijective. Then~${\cL_{Y_\alpha}J=0}$ implies~${Y_\alpha=0}$
for every~${\alpha\in\Om^{2\sn-2}(M)}$.
\end{proposition}

\begin{proof}
Assume~${\cL_{Y_\alpha}J=0}$.  
Then~${\iota(Y_\alpha)\om}$ is harmonic
by~\cite[Lemma~3.9(ii)]{GPST} and is exact 
because~$\iota(Y_\alpha)\om\wedge\om^{\sn-1}/(\sn-1)!=d\alpha$.
Thus~$Y_\alpha=0$.
\end{proof}

By part~(i) of Proposition~\ref{prop:REGULAR} an 
almost complex structure~${J\in\sJ_0(M)}$ is regular if and only if the linear 
map~${\Om^{0,1}_J(M,TM)\to d\Om^1(M):\Jhat\mapsto\Richat_\rho(J,\Jhat)}$
is surjective.  By the implicit function theorem 
in appropriate Sobolev completions this implies that the 
regular part of~$\sJ_0(M,\rho)$ is a 
co-isotropic submanifold of~$\sJ_0(M)$ whose 
isotropic fibers are the $\Diff^\ex(M,\rho)$-orbits.  
One would like to deduce that the regular part 
of~$\sW_0(M,\rho)$ is a symplectic orbifold
with the tangent spaces
$$
T_{[J]_\rho}\sW_0(M,\rho) 
= \frac{\left\{\Jhat\in\Om^{0,1}_J(M,TM)\,|\,\Richat_\rho(J,\Jhat)=0\right\}}
{\left\{\cL_{Y_\alpha}J\,|\,\alpha\in\Om^{2\sn-2}(M)\right\}}
$$
at regular elements~${J\in\sJ_0(M,\rho)}$.
Indeed, the $2$-form~\eqref{eq:OM} is nondegenerate 
on this quotient by part~(ii) of Proposition~\ref{prop:REGULAR}.
However, the action of~${\Diff^\ex(M,\rho)}$ on~$\sJ_0(M,\rho)$ is not always 
proper and the quotient~$\sW_0(M,\rho)$ need not be Hausdorff.  
The archetypal example is the K3-surface~\cite{GHS,VERBITSKY1}.

\begin{example}\label{ex:K3}
Let~$(M,J)$ be a K3-surface that admits 
an embedded holomorphic sphere~$C$ 
with self-intersection number~${C\cdot C=-2}$, 
and let~${\tau:M\to M}$ be the Dehn twist about~$C$.   
Then there exists a smooth family of complex 
structures~$\{J_t\}_{t\in\C}$ and a smooth family 
of diffeomorphisms~${\{\phi_t\in\Diff_0(M)\}_{t\in\C\setminus\{0\}}}$
such that~${J_0=J}$ and~${\phi_t^*J_t=\tau^*J_{-t}}$ for~${t\ne0}$.
Thus the complex structures~$J_t$ and~$\tau^*J_{-t}$ 
represent the same equivalence class in~$\sW_0(M)$, 
however, their limits~${\lim_{t\to0}J_t=J_0}$
and~${\lim_{t\to0}\tau^*J_{-t}=\tau^*J_0}$ do not represent the 
same class in~$\sW_0(M)$ because the homology class~$[C]$ 
belongs to the effective cone of~$J_0$ while the class~$-[C]$ 
belongs to the effective cone of~$\tau^*J_0$.  This shows that
the action of~$\Diff_0(M)$ on~$\sJ_0(M)$ is not proper
and neither is the action of~${\Diff_0(M,\rho)=\Diff^\ex(M,\rho)}$ 
on~$\sJ_0(M,\rho)$.
\end{example}


\subsection{Symplectic Einstein structures}\label{subsec:SE}
Let~$M$ be a closed oriented $2\sn$-mani\-fold and fix a 
nonzero real number~$\hbar$.
A {\bf tame symplectic Einstein structure} on~$M$ is a pair~$(\om,J)$
consisting of a symplectic form~${\om\in\Om^2(M)}$ and 
an almost complex structure~$J$ tamed by~$\om$
(i.e.~$\om(v,Jv)>0$ for~${0\ne v\in TM}$) such that
\begin{equation}\label{eq:SE}
\Ric_{\rho,J} = \om/\hbar,\qquad \rho := \om^\sn/\sn!.
\end{equation}
Every such structure satisfies~${2\pi\hbar c_1^\R(\om)=[\om]}$.   
In the case~${\dim(M)=4}$ the integral first Chern class 
of a symplectic form~$\om$ depends only on the cohomology 
class of~$\om$ (see~\cite[Proposition~13.3.11]{MS}),
while in higher dimensions this is an open question.

The purpose of this subsection is to exhibit the space of equivalence 
classes of tame symplectic Einstein structures in a fixed cohomology 
class~${[\om]=a\in H^2(M;\R)}$ and with a fixed volume 
form~${\om^\sn/\sn!=\rho}$, modulo the action of the group
of exact volume preserving diffeomorphisms, 
as a symplectic quotient.

A pair~$(a,\rho)$ consisting of a cohomology class~${a\in H^2(M;\R)}$
and a positive volume form~$\rho$ is called a {\bf Lefschetz pair} 
if it satisfies the following conditions.

\smallskip\noindent{\bf (V)}
{\it ${V:=\inner{a^\sn/\sn!}{[M]}>0}$ and~${\int_M\rho=V}$.}

\smallskip\noindent{\bf (L)}
{\it The homomorphism~${H^1(M;\R)\to H^{2\sn-1}(M;\R): b\mapsto a^{\sn-1}\cup b}$
is bijective.} 

\smallskip\noindent
Fix a Lefschetz pair~$(a,\rho)$. Then the space
$$
\sS_{a,\rho} := \left\{\om\in\Om^2(M)\,\big|\,d\om=0,\,[\om]=a,\,\om^\sn/\sn!=\rho\right\}
$$
of symplectic forms on~$M$ in the cohomology class~$a$
with volume form~$\rho$ is an infinite-dimensional manifold 
whose tangent space at~${\om\in\sS_{a,\rho}}$ is given by
$$
T_\om\sS_{a,\rho} 
= \left\{\omhat\in d\Om^1(M)\,\big|\,
\omhat\wedge\om^{\sn-1}=0\right\}.
$$
The proof uses the fact that 
the map~${\sS_a\to\sV_a:\om\mapsto\om^\sn/\sn!}$
from the space~$\sS_a$ of symplectic forms in the class~$a$
to the space~$\sV_a$ of volume forms with total volume~$V$ 
is a submersion. (It is surjective by Moser isotopy 
whenever~${\sS_a\ne\emptyset}$.)
It was shown by Trautwein~\cite[Lemma~6.4.2]{TRAUTWEIN} 
that~$\sS_{a,\rho}$ carries a symplectic structure 
\begin{equation}\label{eq:OMS}
\Om_\om(\omhat_1,\omhat_2) 
:= \int_M\lambdahat_1\wedge\lambdahat_2
\wedge\frac{\om^{\sn-1}}{(\sn-1)!}
\end{equation}
for~$\om\in\sS_{a,\rho}$ and~${\omhat_1,\omhat_2\in T_\om\sS_{a,\rho}}$, 
where~${\lambdahat_i\in\Om^1(M)}$ is chosen such that
\begin{equation}\label{eq:lambdahat}
d\lambdahat_i=\omhat_i,\qquad
\lambdahat_i\wedge\om^{\sn-1}\in  d\Om^{2\sn-2}(M).
\end{equation}
Here the existence of~$\lambdahat_i$ 
and the nondegeneracy of~\eqref{eq:OMS} 
both require the Lefschetz condition~(L).
In the rational case a result of Fine~\cite{FINE} shows
that~$\sS_{a,\rho}$ is a symplectic quotient via the action 
of the group of gauge transformations on the space of connections 
with symplectic curvature on a suitable line bundle.

Now consider the space
$$
\sP(M,a,\rho) := \left\{(\om,J)\in\sS_{a,\rho}\times\sJ(M)\,\big|\,
\om(v,Jv)>0\mbox{ for all }0\ne v\in TM\right\}
$$
of all pairs~$(\om,J)$ consisting of a symplectic form~$\om$
in the class~$a$ with volume form~$\rho$ and an $\om$-tame 
almost complex structure~$J$.   This is an open subset 
of the product space~${\sS_{a,\rho}\times\sJ(M)}$,
and the symplectic forms~\eqref{eq:OMS} on~$\sS_{a,\rho}$ 
and~\eqref{eq:OM}~on~$\sJ(M)$ together determine 
a natural product symplectic structure on~$\sP(M,a,\rho)$,
given by 
\begin{equation}\label{eq:OMSE}
\Om_{\om,J}\bigl((\omhat_1,\Jhat_1),(\omhat_2,\Jhat_2)\bigr)
:= \tfrac{1}{2}\int_M\trace\bigl(\Jhat_1J\Jhat_2\bigr)\frac{\om^\sn}{\sn!}
- \tfrac{2}{\hbar}\int_M\lambdahat_1\wedge\lambdahat_2\wedge\frac{\om^{\sn-1}}{(\sn-1)!}
\end{equation}
for~${(\om,J)\in\sP(M,a,\rho)}$ 
and~${(\omhat_i,\Jhat_i)\in T_{(\om,J)}\sP(M,a,\rho)=T_\om\sS_{a,\rho}\times\Om^{0,1}_J(M,TM)}$,
where the~$\lambdahat_i$ are as in~\eqref{eq:lambdahat}.

Throughout we will use the notation~$Y_\alpha$ for the exact divergence-free 
vector field associated to a $(2\sn-2)$-form~${\alpha\in\Om^{2\sn-2}(M)}$
via~${\iota(Y_\alpha)\rho=d\alpha}$.  When the choice of the 
symplectic form~$\om$ is clear from the context, we will use the 
notation~$v_H$ for the Hamiltonian vector field associated
to a function~${H\in\Om^0(M)}$ via~${\iota(v_H)\om = dH}$.

\begin{theorem}[{\bf Trautwein}]\label{thm:TRAUTWEIN}
The symplectic form~\eqref{eq:OMSE} on~$\sP(M,a,\rho)$ 
is preserved by the action of~$\Diff(M,\rho)$. 
The action of the subgroup~$\Diff^\ex(M,\rho)$
is a Hamiltonian group action and is generated by the 
$\Diff(M,\rho)$-equivariant moment map~${\sP(M,a,\rho)\to d\Om^1(M):
(\om,J)\mapsto 2\bigl(\Ric_{\rho,J}-\om/\hbar\bigr)}$, i.e.\ 
\begin{equation}\label{eq:SEMOMENT}
\int_M2\bigl(\Richat_\rho(J,\Jhat)-\omhat/\hbar\bigr)\wedge\alpha
= \Om_{\om,J}\bigl((\omhat,\Jhat),(\cL_{Y_\alpha}\om,\cL_{Y_\alpha}J)\bigr)
\end{equation}
for all~$(\om,J)\in\sP(M,a,\rho)$, all~${(\omhat,\Jhat)\in T_{(\om,J)}\sP(M,a,\rho)}$,
and all~${\alpha\in\Om^{2\sn-2}(M)}$.
\end{theorem}

\begin{proof}
Equation~\eqref{eq:SEMOMENT} follows from
Theorem~\ref{thm:RICCI} and the identity
$$
\int_M\omhat\wedge\alpha
= \int_M\lambdahat\wedge\iota(Y_\alpha)\om\wedge\frac{\om^{\sn-1}}{(\sn-1)!}
= \Om_\om(\omhat,\cL_{Y_\alpha}\om)
$$
for all~${\om\in\sS_{a,\rho}}$, all~${\omhat=d\lambdahat\in T_\om\sS_{a,\rho}}$,
and all~${\alpha\in\Om^{2\sn-2}(M)}$.
\end{proof}

\begin{proposition}\label{prop:PROPER}
The action of~$\Diff^\ex(M,\rho)$ on~$\sP(M,a,\rho)$ is proper.
\end{proposition}

\begin{proof}
The properness proof by Fujiki--Schumacher~\cite{FS1} 
carries over to the present situation.  Choose 
sequences~${(\om_i,J_i)\in\sP(M,a,\rho)}$ and~${\phi_i\in\Diff^\ex(M,\rho)}$
such that the limits~${(\om,J)=\lim_{t\to\infty}(\om_i,J_i)}$
and~${(\om',J')=\lim_{t\to\infty}(\phi_i^*\om_i,\phi_i^*J_i)}$
exist in the $\Cinf$ topology and both belong to~$\sP(M,a,\rho)$.
Define the Riemannian metrics~$g_i$ 
by~${g_i(u,v):=\tfrac{1}{2}(\om_i(u,J_iv)+\om_i(v,J_iu))}$
and similarly for~${g,g'}$.  Then~$g_i$ converges to~$g$ 
and~${\phi_i^*g_i}$ converges to~$g'$ in the $\Cinf$ topology.  
Thus by~\cite[Lemma~3.8]{FS1} a subsequence
of~$\phi_i$ converges to a diffeomorphism~${\phi\in\Diff(M,\rho)}$
in the $\Cinf$ topology.  Now the flux 
homomorphism ${\Flux_\rho:\pi_1(\Diff_0(M,\rho))\to H^{2\sn-1}(M;\R)}$
has a discrete image by~\cite[Exercise~10.2.23(v)]{MS}.
Thus it follows from standard arguments as 
in~\cite[Theorem~10.2.5 \& Proposition~10.2.16]{MS}
that~$\Diff^\ex(M,\rho)$ is a closed subgroup of~$\Diff(M,\rho)$
with respect to the $\Cinf$ topology.
Hence~${\phi\in\Diff^\ex(M,\rho)}$.
\end{proof}

In the present setting the Marsden--Weinstein quotient 
is the space of exact volume preserving isotopy classes 
of tame symplectic Einstein structures given by
\begin{equation}\label{eq:WSE}
\begin{split}
\sW_\SE(M,a,\rho) 
&:= \sP_\SE(M,a,\rho)/\Diff^\ex(M,\rho),\\
\sP_\SE(M,a,\rho) 
&:= \left\{(\om,J)\in\sP(M,a,\rho)\,|\,
\Ric_{\rho,J}=\om/\hbar\right\}.
\end{split}
\end{equation}
The next result is the analogue of Proposition~\ref{prop:REGULAR}
in the symplectic Einstein setting.   In particular, part~(i) asserts that 
a pair~${(\om,J)\in\sP(M,a,\rho)}$ is a regular point for the moment map
if and only if the isotropy subgroup is discrete.   
While the finite-dimensional analogue follows directly from the definitions, 
in the present situation the proof requires elliptic regularity
(in the guise of Proposition~\ref{prop:REGULAR}).

\begin{proposition}\label{prop:SEREGULAR}
Fix a pair~${(\om,J)\in\sP(M,a,\rho)}$. 

\smallskip\noindent{\bf (i)}
Let~${\lambdahat\in\Om^1(M)}$. Then there  exists 
a pair~${(\omhat,\Jhat)\in T_{(\om,J)}\sP(M,a,\rho)}$ that 
satisfies~${\Richat_\rho(J,\Jhat)-\omhat/\hbar=d\lambdahat}$ 
if and only if~${\int_M\lambdahat\wedge\iota(v_H)\rho=0}$ 
for all~${H\in\Om^0(M)}$ with~${\cL_{v_H}J=0}$.

\smallskip\noindent{\bf (ii)}
Let~${(\omhat,\Jhat)\in T_{(\om,J)}\sP(M,a,\rho)}$. 
Then there exists a $(2\sn-2)$-form~${\alpha\in\Om^{2\sn-2}(M)}$ 
that satisfies~${\cL_{Y_\alpha}\om=\omhat}$
and $\cL_{Y_\alpha}J=\Jhat$ if and only 
if $\Om_{\om,J}\bigl((\omhat,\Jhat),(\omhat',\Jhat')\bigr)=0$ 
for all~${(\omhat',\Jhat')\in T_{(\om,J)}\sP(M,a,\rho)}$ 
with~${\Richat_\rho(J,\Jhat')=\omhat'/\hbar}$.
\end{proposition}

The necessity of the conditions in~(i) and~(ii) follows directly from~\eqref{eq:SEMOMENT}.
The proof of the converse implications relies on the following three lemmas,
which allow us to reduce the result to Proposition~\ref{prop:REGULAR}.

\begin{lemma}\label{le:HAM1}
Let~${(\om,J)\in\sP(M,a,\rho)}$ and let~${\alpha\in\Om^{2\sn-2}(M)}$.  
Then~$Y_\alpha$ is Hamiltonian if and only if~${\int_M\lambda\wedge d\alpha=0}$
for every $1$-form~$\lambda$ with~${d\lambda\wedge\om^{\sn-1}=0}$.
\end{lemma}

\begin{proof}
That the condition is necessary follows directly from the definitions.
Conversely, assume that~${\int_M\lambda\wedge d\alpha=0}$
for all~${\lambda\in\Om^1(M)}$ with~${d\lambda\wedge\om^{\sn-1}=0}$.
Let~$\beta$ be a closed $(2\sn-1)$-form and choose~${\lambda\in\Om^1(M)}$
with~${\beta=\lambda\wedge\om^{\sn-1}/(\sn-1)!}$.  Then
$$
\int_M \beta\wedge(({*d\alpha})\circ J)
= \int_M \lambda\wedge(({*d\alpha})\circ J)\wedge\frac{\om^{\sn-1}}{(\sn-1)!}
= \int_M\lambda\wedge d\alpha 
= 0.
$$
Hence~${({*d\alpha})\circ J}$ is an exact $1$-form.  Choose~${H\in\Om^0(M)}$
such that~${({*d\alpha})\circ J=dH}$.  Then~${{*d\alpha}=-dH\circ J}$, 
hence~${d\alpha={*(dH\circ J)}=dH\wedge\om^{\sn-1}/(\sn-1)!=\iota(v_H)\rho}$ 
and so~${Y_\alpha=v_H}$ is a Hamiltonian vector field. 
\end{proof}

\begin{lemma}\label{le:HAM2}
Let~${(\om,J)\in\sP(M,a,\rho)}$ and let~${\sY\subset\Vect^\ex(M,\rho)}$ 
be a finite-di\-men\-sio\-nal subspace that contains no nonzero Hamiltonian 
vector field.  Then, for every linear functional~${\Phi:\sY\to\R}$, there exists 
a $1$-form~$\lambda$ such that~${d\lambda\wedge\om^{\sn-1}=0}$
and~${\int_M\lambda\wedge\iota(Y)\rho=\Phi(Y)}$ for all~${Y\in\sY}$.
\end{lemma}

\begin{proof}
Define $\sL:=\{\lambda\in\Om^1(M)\,|\,d\lambda\wedge\om^{\sn-1}=0\}$
and consider the linear map~${\sL\to\sY^*:\lambda\mapsto\Phi_\lambda}$ 
defined by~${\Phi_\lambda(Y):=\int_M\lambda\wedge\iota(Y)\rho}$ 
for~${\lambda\in\sL}$ and~${Y\in\sY}$.  Then the dual 
map~${\sY\to\sL^*}$ is injective by assumption and Lemma~\ref{le:HAM1}.
Since~$\sY$ is finite-dimensional, this implies that 
the map~${\sL\to\sY^*}$ is surjective.
\end{proof}

\begin{lemma}\label{le:HAM3}
Let~${(\om,J)\in\sP(M,a,\rho)}$ and let~${\alpha\in\Om^{2\sn-2}(M)}$.
Then the following assertions are equivalent.

\smallskip\noindent{\bf (a)}
There exists a function~${H\in\Om^0(M)}$ such that~${\cL_{v_H}J=\cL_{Y_\alpha}J}$.

\smallskip\noindent{\bf (b)}
If~${\lambda\in\Om^1(M)}$ satisfies~${d\lambda\wedge\om^{\sn-1}=0}$
and~${\int_Md\lambda\wedge\beta=0}$ 
for all~${\beta\in\Om^{2\sn-2}(M)}$ with~${\cL_{Y_\beta}J=0}$,
then~${\int_Md\lambda\wedge\alpha=0}$.
\end{lemma}

\begin{proof}
Assume~(a) and define~${\beta:=\alpha-H\om^{\sn-1}/(\sn-1)!}$.
Then~${Y_\beta=Y_\alpha-v_H}$, hence~${\cL_{Y_\beta}J=0}$ by~(a),
and so each~$\lambda$ as in~(b) satisfies
$$
\int_Md\lambda\wedge\alpha
= \int_Md\lambda\wedge\left(\beta+H\frac{\om^{\sn-1}}{(\sn-1)!}\right)
= 0.
$$
Conversely, assume~(a) does not hold and
choose a subspace~${\sY_0\subset\Vect^\ex(M,\rho)}$ such 
that $\{Y_\beta\,|\,\cL_{Y_\beta}J=0\}=\sY_0\oplus\{v_H\,|\,\cL_{v_H}J=0\}$.
Then~${\sY:=\sY_0\oplus\R Y_\alpha}$
does not contain nonzero Hamiltonian vector fields
and so, by Lemma~\ref{le:HAM2} there is a~${\lambda\in\Om^1(M)}$ 
such that~${d\lambda\wedge\om^{\sn-1}=0}$,
${\int_M\lambda\wedge\iota(Y_\alpha)\rho=1}$,
and~${\int_M\lambda\wedge\iota(Y)\rho=0}$ for all~${Y\in\sY_0}$.  
This implies~${\int_Md\lambda\wedge\beta=0}$
for all~${\beta\in\Om^{2\sn-2}(M)}$ with~${\cL_{Y_\beta}J=0}$ 
and~${\int_Md\lambda\wedge\alpha=1}$.  
Hence~(b) does not hold.
\end{proof}

\begin{proof}[Proof of Proposition~\ref{prop:SEREGULAR}]
We prove part~(i). Assume that $\lambdahat\in\Om^1(M)$ 
satisfies~${\int_M\lambdahat\wedge\iota(v_H)\rho=0}$
for all~${H\in\Om^0(M)}$ with~${\cL_{v_H}J=0}$.
Then by Lemma~\ref{le:HAM2} 
there exists a~${\lambda\in\Om^1(M)}$ 
such that~${d\lambda\wedge\om^{\sn-1}=0}$
and~${\int_M(\lambda+\lambdahat)\wedge\iota(Y_\alpha)\rho=0}$
for all~${\alpha\in\Om^{2\sn-2}(M)}$ with~${\cL_{Y_\alpha}J=0}$.
By part~(i) of Proposition~\ref{prop:REGULAR}
there exists a~${\Jhat\in\Om^{0,1}_J(M,TM)}$ 
such that~${\Richat_\rho(J,\Jhat)=d(\lambda+\lambdahat)}$.
This proves~(i) with~${\omhat:=\hbar d\lambda}$. 

To prove~(ii),
let~${(\omhat,\Jhat)\in T_{(\om,J)}\sP(M,a,\rho)}$  
such that~${\Om_{\om,J}\bigl((\omhat,\Jhat),(\omhat',\Jhat')\bigr)=0}$
for all~${(\omhat',\Jhat')\in T_{(\om,J)}\sP(M,a,\rho)}$
with~${\Richat_\rho(J,\Jhat')=\omhat'/\hbar}$.
Then by part~(ii) of Proposition~\ref{prop:REGULAR}
there exist~${\alpha,\beta}$ 
with~${\omhat = d\iota(Y_\alpha)\om}$ and~${\Jhat=\cL_{Y_\beta}J}$.
Let~${\lambda\in\Om^1(M)}$ such that~${d\lambda\wedge\om^{\sn-1}=0}$
and~$\int_M{d\lambda\wedge\alpha'=0}$ for all~${\alpha'\in\Om^{2\sn-2}(M)}$
with~${\cL_{Y_{\alpha'}}J=0}$.  By part~(i) of Proposition~\ref{prop:REGULAR}
choose~$\Jhat'$ such that~${\Richat_\rho(J,\Jhat')=d\lambda}$. 
Then
\begin{equation*}
\begin{split}
2\int_M d\lambda\wedge(\beta - \alpha)
&= 
\tfrac{1}{2}\int_M
\trace\bigl(\Jhat' J\cL_{Y_\beta}J\bigr)\frac{\om^\sn}{\sn!}
- \tfrac{2}{\hbar}\int_M\hbar\lambda
\wedge\bigl(\iota(Y_\alpha)\om\bigr)
\wedge\frac{\om^{\sn-1}}{(\sn-1)!}  \\
&= 
\Om_{\om,J}\bigl((\omhat',\Jhat'),(\omhat,\Jhat)\bigr) 
= 0,
\end{split}
\end{equation*}
where~${\omhat':=\hbar d\lambda=\hbar\Richat_\rho(J,\Jhat')}$.  
Thus by Lemma~\ref{le:HAM3} there exists a 
function~$H$ such that~${\cL_{Y_{\beta-\alpha}}J=\cL_{v_H}J}$,
so~${\cL_{Y_\alpha+v_H}J=\Jhat}$ and~${d\iota(Y_\alpha+v_H)\om=\omhat}$.
This proves~(ii).
\end{proof}

Call a pair~${(\om,J)\in\sP(M,a,\rho)}$ {\bf regular} if there are 
no nonzero $J$-holomorphic Hamiltonian vector fields. 
By part~(i) of Proposition~\ref{prop:SEREGULAR} a pair~$(\om,J)$ is regular 
if and only if the map
$
T_{(\om,J)}\sP(M,a,\rho)\to d\Om^1(M):
(\omhat,\Jhat)\mapsto\Richat_\rho(J,\Jhat)-\omhat/\hbar
$
is surjective.  Thus by Proposition~\ref{prop:PROPER}
(and a suitable local slice theorem that requires a Nash--Moser type proof)
the regular part of~$\sW_\SE(M,a,\rho)$ is a symplectic orbifold
whose tangent space at the equivalence class of a regular 
element~${(\om,J)\in\sP_\SE(M,a,\rho)}$ is the quotient
$$
T_{[\om,J]_\rho}\sW_\SE(M,a,\rho) 
= \frac{\bigl\{(\omhat,\Jhat)\,\big|\,
\omhat\wedge\om^{\sn-1}=0,\,
\Richat_\om(J,\Jhat)=\omhat/\hbar
\bigr\}}
{\left\{(\cL_{Y_\alpha}\om,\cL_{Y_\alpha}J)
\,|\,\alpha\in\Om^{2\sn-2}(M)\right\}}.
$$
The $2$-form~\eqref{eq:OMSE} is nondegenerate on this quotient 
by part~(ii) of Proposition~\ref{prop:SEREGULAR}.

\begin{remark}[{\bf Compatible pairs}]\label{rmk:COMPATIBLE}
The space~${\sC(M,a,\rho)\subset\sP(M,a,\rho)}$ of compatible pairs 
is a submanifold of~$\sP(M,a,\rho)$ and~${(\omhat,\Jhat)\in T_{(\om,J)}\sP(M,a,\rho)}$ 
is tangent to~$\sC(M,a,\rho)$ at a compatible pair~$(\om,J)$
if and only if 
\begin{equation}\label{eq:TC}
\omhat(u,v)-\omhat(Ju,Jv) = \om(\Jhat u,Jv) + \om(Ju,\Jhat v).
\end{equation}
It is an open question whether the restriction 
of the $2$-form~\eqref{eq:OMSE} 
to~$\sC(M,a,\rho)$ is nondegenerate.
A pair~${(\omhat,\Jhat)\in T_{(\om,J)}\sC(M,a,\rho)}$
belongs to its kernel if and only if
\begin{equation}\label{eq:KERNEL}
\Jhat+\Jhat^*=0,\qquad \Richat_\rho(J,\Jhat)=\omhat/\hbar.
\end{equation}
If this holds, then there exists a vector field~${X\in\Vect(M)}$ such 
that~${\omhat=d\iota(X)\om}$ and~${\Jhat = \tfrac{1}{2}(\cL_XJ-(\cL_XJ)^*)}$
(respectively~${\Jhat=\cL_XJ}$ in the K\"ahler--Einstein case). 
Thus the solutions of~\eqref{eq:KERNEL} form the kernel 
of a Fredholm operator, so nondegeneracy is an open condition.  
It follows also that the restriction of the $2$-form~\eqref{eq:OMSE} 
to~$\sC(M,a,\rho)$ is nondegenerate at a K\"ahler--Einstein pair~$(\om,J)$ 
if and only if 
\begin{equation}\label{eq:NONDEG}
\cL_XJ+(\cL_XJ)^*=0\qquad\implies\qquad\cL_XJ=0.
\end{equation}
Now fix a K\"ahler manifold~$(M,\om,J)$.
Then Proposition~\ref{prop:LAMBDA} yields the equation
\begin{equation*}
\begin{split}
\tfrac{1}{2}\INNER{(\cL_XJ)^*}{\cL_XJ}
&= 
-\tfrac{1}{2}\int_M\trace\bigl((\cL_XJ)J(\cL_{JX}J)\bigr)\rho 
=
- \int_M\Lambda_\rho(J,\cL_XJ)\wedge\iota(JX)\rho \\
&= 
\int_M\Bigl(df_X\circ J - df_{JX} - 2\iota(X)\Ric_{\rho,J}\Bigr) \wedge\iota(JX)\rho 
\end{split}
\end{equation*}
for every vector field~$X$, and this implies the Weitzenb\"ock formula 
\begin{equation}\label{eq:WB}
\tfrac{1}{4}\Norm{\cL_XJ+(\cL_XJ)^*}^2
= \tfrac{1}{2}\Norm{\cL_XJ}^2 
+ \int_M\bigl(f_X^2+f_{JX}^2-2\Ric_{\rho,J}(X,JX)\bigr)\rho.
\end{equation}
In the K\"ahler--Einstein case~${\Ric_{\rho,J}=\om/\hbar}$
with~${\hbar<0}$ this identity implies~\eqref{eq:NONDEG}.
In the Fano case~${\hbar>0}$ it is an open question 
whether~\eqref{eq:NONDEG} holds for all K\"ahler--Einstein pairs.
\end{remark}


\subsection{Scalar curvature}\label{subsec:SCALAR}

Let~$(M,\om)$ be a closed symplectic $2\sn$-manifold
with the volume form~${\rho:=\om^\sn/\sn!}$.
For~${F,G\in\Om^0(M)}$ we denote by~$v_F$ the Hamiltonian vector 
field of~$F$ and by~${\{F,G\}:=\om(v_F,v_G)}$ the Poisson bracket.
Let~$\sJ(M,\om)$ be the space of all almost complex structures 
that are compatible with~$\om$, i.e.\ the bilinear 
form~${\inner{\cdot}{\cdot}:=\om(\cdot,J\cdot)}$
is a Riemannian metric.  This is an infinite-dimensional 
manifold whose tangent space at~${J\in\sJ(M,\om)}$ 
is given by
$$
T_J\sJ(M,\om) = \left\{\Jhat\in\Om^{0,1}_J(M,TM)\,\big|\,
\om(J\cdot,\Jhat\cdot)+\om(\Jhat\cdot,J\cdot)=0\right\}.
$$
Here the condition~${\om(J\cdot,\Jhat\cdot)+\om(\Jhat\cdot,J\cdot)=0}$
holds if and only if~$\Jhat$ is symmetric with respect to the Riemannian 
metric~${\inner{\cdot}{\cdot}=\om(\cdot,J\cdot)}$.  
In particular,~$\cL_vJ$ is symmetric for every symplectic vector field~$v$.
The symplectic form on~$\sJ(M,\om)$ is given by
\begin{equation}\label{eq:OMom}
\Om_{\om,J}(\Jhat_1,\Jhat_2) 
:= \tfrac{1}{2}\int_M\trace\bigl(\Jhat_1 J\Jhat_2\bigr)\frac{\om^\sn}{\sn!}
\end{equation}
for~${J\in\sJ(M,\om)}$ and~${\Jhat_i\in T_J\sJ(M,\om)}$
and the complex structure is~${\Jhat\mapsto-J\Jhat}$. 
With these structures~$\sJ(M,\om)$ is an 
infinite-dimen\-sio\-nal K\"ahler manifold.

The group~$\Symp(M,\om)$ of symplectomorphisms acts on~$\sJ(M,\om)$
by K\"ahler isometries. By a theorem of Donaldson~\cite{DON1},
the action of the subgroup~$\Ham(M,\om)$ of Hamiltonian 
symplectomorphisms on~$\sJ(M,\om)$ is a Hamiltonian group action 
and the scalar curvature appears as a moment map. 
Earlier versions of this result were proved by Quillen 
(for Riemann surfaces) and Fujiki~\cite{FUJIKI} (in the integrable case). 
Below we derive it as a direct consequence of Theorem~\ref{thm:RICCI}. 

To explain this, we first observe that the Lie algebra of the 
group~$\Ham(M,\om)$ is the space of Hamiltonian vector fields 
and hence can be identified with the quotient space~$\Om^0(M)/\R$.  
Its dual space can formally be thought of as the space~$\Om^0_\rho(M)$ 
of all functions~$f\in\Om^0(M)$ with mean value zero, 
in that every such function determines a continuous linear 
functional~${\Om^0(M)/\R\to\R:[H]\mapsto\int_MfH\rho}$.
Now the scalar curvature of an almost complex 
structure~${J\in\sJ(M,\om)}$ is defined by
\begin{equation}\label{eq:SCALAR}
S_{\om,J}\rho := 2\Ric_{\rho,J}\wedge\om^{\sn-1}/(\sn-1)!.
\end{equation}
Its mean value is the topological invariant
$
c_\om := \frac{4\pi}{V}\inner{c_1(\om)\cup\frac{[\om]^{\sn-1}}{(\sn-1)!}}{[M]},
$
${V:=\int_M\rho}$. 
In the integrable case $S_{\om,J}$ is the standard scalar curvature 
of the K\"ahler metric.

\begin{theorem}[{\bf Quillen--Fujiki--Donaldson}]\label{thm:SCALAR}
The action of~$\Ham(M,\om)$ on the space~$\sJ(M,\om)$ is Hamiltonian 
and is generated by the~$\Symp(M,\om)$-equivariant moment 
map~${\sJ(M,\om)\to\Om^0_\rho(M):J\mapsto S_{\om,J}-c_\om}$, i.e.\ 
\begin{equation}\label{eq:SCALARMOMENT}
\int_M\Shat_\om(J,\Jhat)H\frac{\om^\sn}{\sn!} 
= \Om_{\om,J}(\Jhat,\cL_{v_H}J)
\end{equation}
for all~${J\in\sJ(M,\om)}$, all~${\Jhat\in T_J\sJ(M,\om)}$, 
and all~${H\in\Om^0(M)}$, where
$$
\Shat_\om(J,\Jhat) := \left.\tfrac{d}{dt}\right|_{t=0}S_{\om,J_t}
$$
for any smooth path~${\R\to\sJ(M,\om):t\mapsto J_t}$ 
satisfying~${J_0=J}$ and~${\left.\tfrac{d}{dt}\right|_{t=0}J_t=\Jhat}$. 
\end{theorem}

\begin{proof}
By~\eqref{eq:SCALAR} we have
\begin{equation*}
\begin{split}
\int_M\Shat_\om(J,\Jhat)H\frac{\om^\sn}{\sn!}
=
\int_M 2\Richat_\rho(J,\Jhat)\wedge H\frac{\om^{\sn-1}}{(\sn-1)!} 
= 
\tfrac{1}{2}\int_M\trace\bigl(\Jhat J\cL_{v_H}J\bigr)\frac{\om^\sn}{\sn!}.
\end{split}
\end{equation*}
The last equation follows from Theorem~\ref{thm:RICCI}
with~${Y_\alpha=v_H}$ and~${\alpha=H\tfrac{\om^{\sn-1}}{(\sn-1)!}}$.
\end{proof}

In the present situation one proves exactly as in Proposition~\ref{prop:PROPER}
that the action of the group~$\Ham(M,\om)$ on~$\sJ(M,\om)$ is proper.
Here the argument uses a theorem of Ono~\cite{ONO} which asserts 
that~$\Ham(M,\om)$ is a closed subgroup of~$\Symp(M,\om)$ with respect
to the $\Cinf$ topology.   Now the Marsden--Weinstein quotient is
the space of Hamiltonian isotopy classes of $\om$-compatible
almost complex structures with constant scalar curvature
given by
\begin{equation}\label{eq:WSCALAR}
\begin{split}
\sW_\csc(M,\om) 
&:= 
\sJ_\csc(M,\om)/\Ham(M,\om),\\
\sJ_\csc(M,\om) 
&:= 
\left\{J\in\sJ(M,\om)\,|\,S_{\om,J}=c_\om\right\}.
\end{split}
\end{equation}
The next result is the analogue of Proposition~\ref{prop:REGULAR}
in the present setting.

\begin{proposition}\label{prop:SREGULAR}
Fix an element~${J\in\sJ(M,\om)}$. 

\smallskip\noindent{\bf (i)}
Let~${f\in\Om^0(M)}$. Then there exists a~${\Jhat\in T_J\sJ(M,\om)}$ 
such that~${\Shat_\om(J,\Jhat)=f}$ if and only if~${\int_M fH\rho=0}$
for all~${H\in\Om^0(M)}$ with~${\cL_{v_H}J=0}$.

\smallskip\noindent{\bf (ii)}
Let~${\Jhat\in T_J\sJ(M,\om)}$. Then there exists an~${H\in\Om^0(M)}$ 
such that~${\cL_{v_H}J=\Jhat}$ if and only if~${\Om_{\om,J}(\Jhat,\Jhat')=0}$
for all~${\Jhat'\in T_J\sJ(M,\om)}$ with~${\Shat_\om(J,\Jhat')=0}$.
\end{proposition}

\begin{proof}
That the conditions in~(i) and~(ii) are necessary follows  
from~\eqref{eq:SCALARMOMENT}.   
To prove the converse, define the operator~${\bL:\Om^0(M)\to\Om^0(M)}$ by 
\begin{equation}\label{eq:LS1}
\begin{split}
\bL F
&:= 
\Shat_\om(J,J\cL_{v_F}J) 
=
-d^*\bigl(\Lambda_\rho(J,J\cL_{v_F}J)\circ J\bigr) \\
&\phantom{:}=
d^*dd^*dF  - 2d^*\bigl(\iota(Jv_F)\Ric_{\rho,J}\circ J\bigr)
+ d^*\bigl(\Lambda_\rho\bigl(J,N_J(v_F,\cdot)\bigr)\circ J\bigr)
\end{split}
\end{equation}
for~${F\in\Om^0(M)}$.  Here the last equality follows by 
a calculation which uses the identities~${f_{Jv_F}=-d^*dF}$ 
and~${N_J(u,v)=J(\cL_vJ)u-(\cL_{Jv}J)u}$ for the Nijenhuis tensor.
The operator~$\bL$ is a fourth order self-adjoint Fredholm operator 
and, by~\eqref{eq:SCALARMOMENT},
\begin{equation}\label{eq:LS2}
\begin{split}
\int_M(\bL F)G\rho = \tfrac{1}{2}\int_M\trace\bigl((\cL_{v_F}J)(\cL_{v_G}J)\bigr)\rho
\end{split}
\end{equation}
for all~${F,G\in\Om^0(M)}$.  
Thus~${H\in\ker\bL}$ if and only if~${\cL_{v_H}J=0}$. 
Hence, if~${f\in\Om^0(M)}$ satisfies~${\int_MfH\rho=0}$ for 
all~$H$ with~${\cL_{v_H}J=0}$, then~${f\in\im\bL}$,
and this proves~(i).

To prove part~(ii), assume that~${\Jhat\in T_J\sJ(M,\om)}$ 
satisfies~${\Om_{\om,J}(\Jhat,\Jhat')=0}$ for all~${\Jhat'\in T_J\sJ(M,\om)}$
with~${\Shat_\om(J,\Jhat')=0}$.
Then, by part~(ii) of Proposition~\ref{prop:REGULAR}, 
there exists an~$\alpha$ such that~${\cL_{Y_\alpha}J=\Jhat}$.
Now let~${\lambda\in\Om^1(M)}$ such 
that~${d\lambda\wedge\om^{\sn-1} = 0}$
and~${\int_Md\lambda\wedge\beta=0}$ 
for all~${\beta\in\Om^{2\sn-2}(M)}$ with~${\cL_{Y_\beta}J=0}$.  
Choose~${\Jhat'\in\Om^{0,1}_J(M,TM)}$ 
with~${\Richat_\rho(J,\Jhat')=d\lambda}$
by part~(i) of Proposition~\ref{prop:REGULAR}.
Then~$\Shat_\om(J,\Jhat')=0$ and hence
$$
2\int_Md\lambda\wedge\alpha
= \int_M2\Richat_\rho(J,\Jhat')\wedge\alpha
= \Om_{\rho,J}(\Jhat',\cL_{Y_\alpha}J) 
= \Om_{\om,J}(\Jhat',\Jhat) 
= 0.
$$
Thus by Lemma~\ref{le:HAM3} there exists an~${H\in\Om^0(M)}$ 
such that~${\cL_{v_H}J=\cL_{Y_\alpha}J=\Jhat}$. 
\end{proof}

\begin{proposition}\label{prop:HARMONIC}
Let~${J\in\sJ(M,\om)}$ and let~${\Jhat\in T_J\sJ(M,\om)}$.  
Then there exists a function~${H\in\Om^0(M)}$
such that~${\Shat_\om(J,\Jhat-J\cL_{v_H}J)=0}$.
Moreover,~$\cL_{v_H}J$ is uniquely determined by this condition.
\end{proposition}

\begin{proof}
Define~${f\in\Om^0(M)}$ by~${f\rho:=d\Lambda_\rho(J,\Jhat)\wedge\om^{\sn-1}/(\sn-1)!}$
and let~$\bL$ be as in Proposition~\ref{prop:SREGULAR}.  Then
$
(\bL H)\rho = d\Lambda_\rho(J,J\cL_{v_H}J)\wedge\om^{\sn-1}/(\sn-1)!
$
and, by~\eqref{eq:SCALARMOMENT},
\begin{equation*}
\begin{split}
\int_MfH\rho
= 
\int_Md\Lambda_\rho(J,\Jhat)\wedge H\frac{\om^{\sn-1}}{(\sn-1)!} 
= 
\tfrac{1}{2}\int_M\trace\bigl(\Jhat J\cL_{v_H}J\bigr)\rho
=
0
\end{split}
\end{equation*}
for all~${H\in\ker\bL}$.  Thus~$f$ belongs to the image of~$\bL$. 
\end{proof}

Call an almost complex structure~${J\in\sJ(M,\om)}$ {\bf regular} 
if there are no nonzero $J$-holomorphic Hamiltonian vector fields.
By part~(i) of Proposition~\ref{prop:SREGULAR}
$J$ is regular if and only if the map
$
T_J\sJ(M,\om)\to\Om^0_\rho(M):\Jhat\mapsto\Shat_\om(J,\Jhat)
$
is surjective. Hence, since the action is proper, 
it follows again from a suitable local slice theorem that
the regular part of the quotient~$\sW_\csc(M,\om)$ is a K\"ahler orbifold.
It is infinite-dimensional when~${\dim(M)>2}$,
and its tangent space at the equivalence class 
of a regular element~${J\in\sJ_\csc(M,\om)}$ is the quotient
\begin{equation}\label{eq:TWom}
\begin{split}
T_{[J]_\om}\sW_\csc(M,\om) 
&= \frac{\left\{\Jhat\in\Om^{0,1}_J(M,TM)\,|\,
\Jhat=\Jhat^*,\,\Shat_\om(J,\Jhat)=0\right\}}
{\left\{\cL_{v_H}J\,|\,H\in\Om^0(M)\right\}}.
\end{split}
\end{equation} 
The $2$-form~\eqref{eq:OMom} is nondegenerate on this quotient 
by part~(ii) of Proposition~\ref{prop:SREGULAR} 
and the complex structure is given 
by~${[\Jhat]_\om\mapsto[-J(\Jhat-\cL_{v_H}J)]_\om}$,
where~$H$ is chosen as in Proposition~\ref{prop:HARMONIC} 
so that~${\Shat_\om(J,J\Jhat-J\cL_{v_H}J)=0}$.

\begin{corollary}\label{cor:MATSUSHIMA}
Let~${J\in\sJ(M,\om)}$ and~${F,G\in\Om^0(M)}$. 
Then
\begin{equation}\label{eq:OMSFG}
\Om_{\rho,J}(\cL_{v_F}J,\cL_{v_G}J) 
= \int_MS_{\om,J}\{F,G\}\frac{\om^\sn}{\sn!}.
\end{equation}
In particular, if~${S_{\om,J}=c_\om}$, then the $L^2$-norm 
of the endomorphism~$\cL_{v_F}J+J\cL_{v_G}J$ is given by
${\Norm{\cL_{v_F}J+J\cL_{v_G}J}^2 
= \Norm{\cL_{v_F}J}^2+\Norm{\cL_{v_G}J}^2}$.
\end{corollary}

\begin{proof}
Let~$\phi_t$ be the flow of~$v_F$.  
Then~${S_{\om,\phi_t^*J}=S_{\om,J}\circ\phi_t}$ for all~$t$ 
and hence differentiation with respect to~$t$ yields the identity
\begin{equation}\label{eq:ShatF}
\Shat_\om(J,\cL_{v_F}J)=\{S_{\om,J},F\}.
\end{equation}
Insert equation~\eqref{eq:ShatF} into~\eqref{eq:SCALARMOMENT}  
with~${\Jhat=\cL_{v_F}J}$ and~${H=G}$ to obtain~\eqref{eq:OMSFG}.
\end{proof}

Let~$(M,\om,J)$ be a closed K\"ahler manifold with
constant scalar curvature such that~${H^1(M;\R)=0}$.
Then every holomorphic vector field is the sum 
of a Hamiltonian and a gradient vector field 
by~\cite[Lemma~3.7(ii)]{GPST},
and for all~${F,G\in\Om^0(M)}$ we
have~${\Norm{\cL_{v_F+Jv_G}J}^2
=\Norm{\cL_{v_F}J}^2+\Norm{\cL_{v_G}J}^2}$
by Corollary~\ref{cor:MATSUSHIMA}.
Hence the Lie algebra of holomorphic vector fields is the complexification 
of the Lie algebra of Killing fields and is therefore reductive.
This is the content of {\bf Matsushima's Theorem}.

\begin{remark}\label{rmk:NOTRICCIFLAT}
Non-integrable almost complex structures in~$\sJ(M,\om)$
with vanishing scalar curvature need not be Ricci-flat.
To see this, assume~${\dim(M)\ge4}$ 
and~${H^1(M;\R)=0}$, and that
there exists a~${J\in\sJ_\INT(M,\om)}$ 
such that~${\Ric_{\rho,J}=0}$ with~${\rho:=\om^\sn/\sn!}$.  
Then there exists a~${\Jhat\in\Om^{0,1}_J(M,TM)}$ such that
\begin{equation}\label{eq:NOTRICCIFLAT}
\Jhat=\Jhat^*,\qquad
\Shat_\om(J,\Jhat)=0,\qquad 
\Richat_\rho(J,\Jhat)\ne 0.
\end{equation}
Namely, choose a nonzero $1$-form~$\lambdahat$ 
such that~${d^*\lambdahat=0}$ and~${d^*(\lambdahat\circ J)=0}$.
Then~$\lambdahat$ is not closed and there is no nonzero $J$-holomorphic 
vector field by~\cite[Lemma~3.9]{GPST}. 
Thus, by~\eqref{eq:WB}  with~${\Ric_{\rho,J}=0}$
the operator~${X\mapsto\cL_XJ+(\cL_XJ)^*}$ is injective 
and hence, by the closed image theorem
the dual operator~${\Jhat=\Jhat^*\mapsto\bar\p_J^*\Jhat}$ is surjective. 
Thus, by Remark~\ref{rmk:LAMBDA} there exists 
a~${\Jhat=\Jhat^*\in\Om^{0,1}_J(M,TM)}$
such that~${\Lambda_\rho(J,\Jhat)=\lambdahat}$.
This implies~${\Shat_\om(J,\Jhat)=-d^*(\lambdahat\circ J)}$
and~${2\Richat_\rho(J,\Jhat)=d\lambdahat}$,
so~$\Jhat$ satisfies~\eqref{eq:NOTRICCIFLAT}.

By~\eqref{eq:NOTRICCIFLAT} there exists a smooth 
curve~${\R\to\sJ(M,\om):t\mapsto J_t}$ such that~${J_0=J}$ 
and $\p_t|_{t=0}J_t=\Jhat$ and $S_{\om,J_t}=0$ for all~$t$.  
Since $\Richat_\rho(J,\Jhat)\ne0$, this curve also 
satisfies~${\Ric_{\rho,J_t}\ne 0}$ for small nonzero~$t$. 
Note that~${\Richat_\rho(J,\Jhat)}$ is not a $(1,1)$-form,
hence~${\bar\p_J\Jhat\ne0}$ by~\cite[Lemma~3.6]{GPST},
and so~$J_t$ is not integrable for small nonzero~$t$.
\end{remark}


\section{Teichm\"uller spaces}\label{sec:TEICH}

In this section we consider integrable complex structures and 
examine the Teich\-m\"uller spaces of Calabi--Yau structures,
of K\"ahler--Einstein structures, 
and of constant scalar curvature K\"ahler metrics.
Such Teichm\"uller spaces have been studied by many authors, see 
e.g.~\cite{BCS,CATANESE,FUJIKI,FS1,FS2,GHS,HUY,KODAIRA,
KOISO,M,N,S3,S4,SIU2,TIAN,TODOROV,VIEHWEG} 
and the references therein.  
The regular part of each Teichm\"uller space 
is a finite-dimensional symplectic submanifold 
of the relevant symplectic quotient in Section~\ref{sec:RICCI}.
It thus acquires a natural symplectic structure
that descends to the Weil--Petersson form on the corresponding moduli space 
(the quotient of Teichm\"uller space by the mapping class group).

In our formulation the ambient manifold~$M$ is fixed 
and the Weil--Petersson form arises via symplectic reduction 
of a Hamiltonian group action by an infinite-dimensional group 
on an infinite-dimensional space with a finite-dimensional quotient.
In the original algebro-geometric approach the moduli space
is directly characterized in finite-dimensional terms via a Torelli
type theorem and the Weil--Petersson form arises from the natural
homogeneous symplectic form on the relevant period domain.

It seems to be an open question whether there exist closed K\"ahler manifolds
that admit holomorphic diffeomorphisms that are smoothly isotopic to the identity,
but not through holomorphic diffeomorphisms. (For nonK\"ahler examples see~\cite{M}.)
If they do exist, then the regular parts of the Teichm\"uller
spaces examined here are orbifolds rather than manifolds.
When discussing Teichm\"uller spaces as manifolds,
we tacitly assume that such automorphisms do not exist,
as is the case for Riemann surfaces.


\subsection{The Teichm\"uller space of Calabi--Yau structures}
\label{subsec:TEICHCY}
Let~$M$ be a closed connected oriented $2\sn$-manifold.
Then the {\bf Teichm\"uller space of Calabi--Yau structures} 
on~$M$ is the space of isotopy classes of complex 
structures with real first Chern class zero and 
nonempty K\"ahler cone.  It is denoted by
\begin{equation}\label{eq:TEICH}
\begin{split}
\sT_0(M) 
&:= \sJ_{\INT,0}(M)/\Diff_0(M),\\ 
\sJ_{\INT,0}(M)
&:=
\bigl\{J\in\sJ_\INT(M)\,\big|\,
c_1^\R(J)=0\mbox{ and }J\mbox{ admits a K\"ahler form}
\bigr\}.
\end{split}
\end{equation}
Associated to a complex structure~${J\in\sJ_{\INT,0}(M)}$ 
there is the Dolbeault complex
\begin{equation}\label{eq:DOLBEAULT}
\Om^0(M,TM)\stackrel{\bar\p_J}{\longrightarrow}
\Om^{0,1}_J(M,TM)\stackrel{\bar\p_J}{\longrightarrow}
\Om^{0,2}_J(M,TM),
\end{equation}
where the first operator corresponds to the infinitesimal action 
of the vector fields on~${T_J\sJ(M)=\Om^{0,1}_J(M,TM)}$
by Remark~\ref{rmk:LAMBDA}, and the second operator 
corresponds to the derivative of the map which assigns to an
almost complex structure~$J$ its Nijenhuis tensor~$N_J$
by~\cite[(3.2)]{GPST}.  Thus the tangent space 
of the Teichm\"uller space~$\sT_0(M)$ at the equivalence class 
of an element~${J\in\sJ_{\INT,0}(M)}$ can formally be identified
with the cohomology of the Dolbeault complex~\eqref{eq:DOLBEAULT}, i.e.\ 
\begin{equation}\label{eq:TTEICH}
T_{[J]}\sT_0(M) 
= \frac{\ker(\bar\p_J:\Om^{0,1}_J(M,TM)\to\Om^{0,2}_J(M,TM))}
{\im(\bar\p_J:\Om^0(M,TM)\to\Om^{0,1}_J(M,TM))}.
\end{equation}
The proof requires a local slice theorem for the action of the 
diffeomorphism group on the space of integrable complex structures. 

For every~${J\in\sJ_{\INT,0}(M)}$ the space of holomorphic vector fields 
is isomorphic to the space of harmonic $1$-forms by~\cite[Lemma~3.9]{GPST}.
Moreover, the Bogomolov--Tian--Todorov theorem asserts
that the obstruction class vanishes~\cite{BOGOMOLOV,TIAN,TODOROV},
so the cohomology of the Dolbeault complex~\eqref{eq:DOLBEAULT}
has constant dimension, and that~$\sT_0(M)$ is indeed 
a smooth manifold whose tangent space at the equivalence class 
of~${J\in\sJ_{\INT,0}(M)}$ is the cohomology group~\eqref{eq:TTEICH}.
The Teichm\"uller space is in general not Hausdorff, even for the K3 surface
(see~\cite{GHS,VERBITSKY1} and also Example~\ref{ex:K3}).
For hyperK\"ahler manifolds the Teichm\"uller space becomes 
Hausdorff after identifying inseparable complex structures 
(see Verbitsky~\cite{VERBITSKY1,VERBITSKY2})
which are biholomorphic by a theorem of Huybrechts~\cite{HUY2}.

Now fix a positive volume form~${\rho\in\Om^{2\sn}(M)}$.  
Then another description of the Teichm\"uller space
of Calabi--Yau structures is as the quotient
\begin{equation}\label{eq:TEICHRHO}
\begin{split}
\sT_0(M,\rho)
&:=
\sJ_{\INT,0}(M,\rho)/\Diff_0(M,\rho)
= \sJ_{\INT,0}(M,\rho)/\Diff^\ex(M,\rho), \\
\sJ_{\INT,0}(M,\rho)
&:=
\bigl\{J\in\sJ_{\INT,0}(M)\,|\,\Ric_{\rho,J}=0\bigr\}.
\end{split}
\end{equation}
Here the two quotients agree
because the quotient group~$\Diff_0(M,\rho)/\Diff^\ex(M,\rho)$
acts trivially on~${\sJ_{\INT,0}(M,\rho)/\Diff^\ex(M,\rho)}$
(see~\cite[Lemma~3.9]{GPST}).  
The tangent space of~$\sT_0(M,\rho)$ at the equivalence class 
of~$J\in\sJ_{\INT,0}(M,\rho)$ is the quotient
\begin{equation}\label{eq:TTEICHRHO}
T_{[J]_\rho}\sT_0(M,\rho) 
= \frac{\bigl\{\Jhat\in\Om^{0,1}_J(M,TM)\,|\,
\bar\p_J\Jhat=0,\,\Richat_\rho(J,\Jhat)=0\bigr\}}
{\bigl\{\cL_{Y_\alpha}J\,|\,\alpha\in\Om^{2\sn-2}(M)\bigr\}}.
\end{equation}
The inclusion~${\iota_\rho:\sT_0(M,\rho)\to\sT_0(M)}$
is a diffeomorphism by~\cite[Lemma~4.2]{GPST}.
The proof uses Moser isotopy and the formula~\eqref{eq:RICf}.
The Teichm\"uller space~$\sT_0(M,\rho)$ is a submani\-fold of
the infinite-dimensional symplectic quotient~$\sW_0(M,\rho)$ 
in~\eqref{eq:WZERO} and it turns out that the $2$-form~\eqref{eq:OM} 
descends to a symplectic form on~$\sT_0(M,\rho)$
and thus induces a symplectic form on~$\sT_0(M)$.
An explicit formula for this symplectic form 
relies on the following two observations.
First, it follows from~\eqref{eq:RICf} that,
for every~${J\in\sJ_{\INT,0}(M)}$, there exists a unique 
positive volume form~${\rho_J\in\Om^{2\sn}(M)}$ that satisfies
\begin{equation}\label{eq:rhoJ}
\Ric_{\rho_J,J}=0,\qquad\int_M\rho_J = V := \int_M\rho.
\end{equation}
Second, for every~${J\in\sJ_{\INT,0}(M)}$ 
and every~${\Jhat\in\Om^{0,1}_J(M,TM)}$
with~${\bar\p_J\Jhat=0}$, there exist unique 
functions~${f,g\in\Om^0(M)}$ that satisfy
\begin{equation}\label{eq:LAMBDAfg}
\Lambda_{\rho_J}(J,\Jhat) = -df\circ J + dg,\qquad
\int_Mf\rho_J = \int_Mg\rho_J=0.
\end{equation}
(See~\cite[Lemma~3.8]{GPST}.)
With this understood, define
\begin{equation}\label{eq:OMWP}
\Om_J^\WP(\Jhat_1,\Jhat_2) 
:= \int_M\Bigl(
\tfrac{1}{2}\trace\bigl(\Jhat_1J\Jhat_2\bigr)
- f_1g_2 + f_2g_1
\Bigr)\rho_J
\end{equation}
for~${J\in\sJ_{\INT,0}(M)}$ and~$\Jhat_i\in\Om^{0,1}_J(M,TM)$
with~${\bar\p_J\Jhat_i=0}$, where~${f_i,g_i}$ are as in~\eqref{eq:LAMBDAfg}.

\begin{theorem}[\cite{GPST}]\label{thm:WPCY}
Equation~\eqref{eq:OMWP} defines a closed $2$-form~$\Om^\WP$ on~$\sJ_{\INT,0}(M)$
that descends to a symplectic form, still denoted by~$\Om^\WP$, 
on the Teichm\"uller space~$\sT_0(M)$.
Its pullback under the diffeomorphism~${\iota_\rho:\sT_0(M,\rho)\to\sT_0(M)}$
is the symplectic form induced by~\eqref{eq:OM} and renders~$\sT_0(M,\rho)$
into a symplectic submanifold of the infinite-dimensional 
symplectic quotient~$\sW_0(M,\rho)$ in~\eqref{eq:WZERO}.
\end{theorem}

Here is why this result is not quite as obvious as it may seem at first glance.  
The space~$\sJ_0(M)$ admits a complex structure
$$
\Jhat\mapsto-J\Jhat
$$
and the symplectic form~$\Om_\rho$ in equation~\eqref{eq:OM} 
is a $(1,1)$-form for this complex structure.  However, when~${\dim(M)>2}$,
it is not a K\"ahler form because the symmetric bilinear form
$$
\Om_{\rho,J}(\Jhat_1,-J\Jhat_2)= \tfrac{1}{2}\int_M\trace(\Jhat_1\Jhat_2)\rho
$$
is indefinite on~${\Om^{0,1}_J(M,TM)}$.
Thus a complex submanifold of~$\sJ_0(M)$ need not be symplectic.  
This is precisely the case for the submanifold~${\sJ_{\INT,0}(M)}$ 
in~\eqref{eq:TEICH}, because the restriction of the 
$2$-form~$\Om_{\rho,J}$ to~${T_J\sJ_{\INT,0}(M)=\ker\bar\p_J}$ 
has a nontrivial kernel in the case~${\dim(M)>2}$,
which in the case~${\Ric_{\rho,J}=0}$ consists of all 
infinitesimal deformations~${\Jhat=\cL_XJ}$ of complex structures
such that both~$X$ and~$JX$ are divergence-free.
A key ingredient in the proof of this assertion is the fact that the space 
of all $\bar\p_J$-harmonic $(0,1)$-forms~${\Jhat\in\Om^{0,1}_J(M,TM)}$
on a closed Ricci-flat K\"ahler manifold is~invariant 
under the homomorphism~${\Jhat\mapsto\Jhat^*}$
(see~\cite[Lemma~3.10]{GPST}).
It then follows that the kernel of the $2$-form~$\Om^\WP_J$
in~\eqref{eq:OMWP} on~${T_J\sJ_{\INT,0}(M)=\ker\bar\p_J}$ 
is the image of~$\bar\p_J$ and hence~$\Om^\WP_J$ 
descends to a nondegenerate $2$-form on the 
tangent space~${T_{[J]}\sT_0(M)=\ker\bar\p_J/\im\bar\p_J}$
for each~${J\in\sJ_{\INT,0}(M)}$.
This shows that~$\Om^\WP$ descends to a symplectic form 
on~$\sT_0(M)$ (see~\cite[Theorem~4.4]{GPST}).

Theorem~\ref{thm:WPCY} gives an alternative construction of 
the Weil--Petersson symplectic form on the
Teichm\"uller space of Calabi--Yau structures 
(see~\cite{HUY,KOISO,N,S4,TIAN,TODOROV}
for the polarized case and~\cite[Ch~16]{FHW}
for the K3 surface).   The Teich\-m\"uller space~$\sT_0(M)$
carries a natural complex structure
$$
[\Jhat]\mapsto[-J\Jhat]
$$
and the Weil--Petersson symplecic form~$\Om^\WP$ 
is of type~$(1,1)$, however, it is not a K\"ahler form in general. 
The complex dimension of the negative part is the Hodge number~$h^{2,0}$
and the total dimension is the Hodge number~${h^{\sn-1,1}(M,L)}$,
where~${L=\Lambda^{n,0}_JTM}$.    
These Hodge numbers are deformation invariant by~\cite[Proposition~9.30]{VOISIN}.

If~${a\in H^2(M;\R)}$ is a K\"ahler class, then the tangent spaces
of the polarized Teich\-m\"uller space
$
\sT_{0,a}(M)\subset\sT_0(M)
$
in Remark~\ref{rmk:TEICHomrho} below are posi\-tive subspaces
for the Weil--Petersson symplectic form and so~$\Om^\WP$ 
restricts to a K\"ahler form on~$\sT_{0,a}(M)$. 
If~${h^{2,0}=0}$, then~$\sT_0(M)$ is Hausdorff and K\"ahler 
and each polarized space~$\sT_{0,a}(M)$ 
is an open subset of~$\sT_0(M)$.
Here is a list of the real dimensions 
for the $2\sn$-torus, 
the K3 surface, 
the Enriques surface, 
the quintic in~$\CP^4$, and
the banana manifold~${\it B}$ in~\cite{BRYAN}.
The last column lists the dimensions of the K\"ahler cones.
\begin{center}
\begin{tabular}{||c||c|c|c|c||}
\hline
\multirow{2}{*}{$M$}  & $\sT_0(M)$ & $\sT_{0,a}(M)$ &  & $\cK_J$  \\
& $2h^{\sn-1,1}(M,L)$ & $2h^{\sn-1,1}(M,L)-2h^{2,0}$ & $2h^{2,0}$ & $h^{1,1}$ \\
 \hline
 \hline
 $\T^{2\sn}$   & $2\sn^2$ & $\sn^2+\sn$ & $\sn^2-\sn$ & $\sn^2$ \\
 ${\it K3}$ & 40 & 38 & 2 & 20 \\
 Enriques & 20 & 20 & 0 & 10 \\
 Quintic    & 202 & 202 & 0 & 1 \\
 ${\it B}$  & 16 & 16 & 0 & 20 \\ 
\hline
\end{tabular}
\end{center}

 
\subsection{The Teichm\"uller space of K\"ahler--Einstein structures}
\label{subsec:TEICHKE}

Let~$M$ be a closed connected oriented $2\sn$-manifold,
let~${c\in H^2(M;\R)}$ be a nonzero co\-ho\-mo\-lo\-gy class 
that admits an integral lift, and let~$\hbar$ be a real number
such that~${(2\pi\hbar c)^\sn>0}$.  Consider the space
$$
\sJ_{\INT,c}(M) := \left\{J\in\sJ_\INT(M)\,\big|\,
c_1^\R(J)=c,\, 2\pi\hbar c\in\cK_J\right\}
$$
of all complex structures~$J$ on~$M$ whose real first Chern class is~$c$
and whose K\"ahler cone~$\cK_J$ contains the cohomology class~${2\pi\hbar c}$.  
For~${J\in\sJ_{\INT,c}(M)}$ denote by 
\begin{equation}\label{eq:SJ}
\sS_J := \left\{\om\in\Om^2(M)\,|\,d\om=0,\,\om^\sn>0,\,
[\om]=2\pi\hbar c,\,J\in\sJ_\INT(M,\om)\right\}
\end{equation}
the space of all symplectic forms~$\om$ on~$M$ that are compatible 
with~$J$ and represent the cohomology class~${2\pi\hbar c}$.
By the Calabi--Yau Theorem~\cite{CALABI2,YAU0} 
each volume form~${\rho\in\Om^{2\sn}(M)}$ 
with~${\int_M\rho=\inner{(2\pi\hbar c)^\sn/\sn!}{[M]}}$
and each~${J\in\sJ_{\INT,c}(M)}$
determine a unique symplectic form~${\om_{\rho,J}\in\sS_J}$ 
whose volume form is~$\rho$, i.e.\ 
\begin{equation}\label{eq:OMRHOJ}
\om_{\rho,J}^\sn/\sn!=\rho.
\end{equation}
In the case of general type with~${\hbar<0}$ a theorem 
of Yau~\cite{YAU1,YAU2} asserts that every complex 
structure~${J\in\sJ_{\INT,c}(M)}$ admits a unique 
symplectic form~${\om\in\sS_J}$ that satisfies the
K\"ahler--Einstein condition
$$
\Ric_{\om^\sn/\sn!,J}=\om/\hbar.
$$
In the Fano case with~${\hbar>0}$ the Chen--Donaldson--Sun 
Theorem~\cite{CDS0,CDS123}
asserts that a complex structure~${J\in\sJ_{\INT,c}(M)}$ 
admits a symplectic form~${\om\in\sS_J}$ that satis\-fies the 
K\"ahler--Einstein condition~${\Ric_{\om^\sn/\sn!,J}=\om/\hbar}$ 
if and only if it satis\-fies the K-polystablility condition 
of Yau--Tian--Donaldson~\cite{DON4,TIAN0,YAU2}. 
The {\it ``only if''} statement was proved earlier by Berman~\cite{BERMAN}.
Moreover, it was shown by Berman--Berndtsson~\cite{BB}
that, if~${\om,\om'\in\sS_J}$ both satisfy the K\"ahler--Einstein condition, 
then there exists a holomorphic diffeomorphism~${\psi\in\Aut_0(M,J)}$
such that~${\om'=\psi^*\om}$. 

Fix a volume form~${\rho\in\Om^{2\sn}(M)}$ 
with~${\int_M\rho=\inner{(2\pi\hbar c)^\sn/\sn!}{[M]}}$
and consider the following models for the 
{\bf Teichm\"uller space of K\"ahler--Einstein structures}:
\begin{equation}\label{eq:TEICHKE}
\begin{split}
\sT_c(M,\rho) 
&:= \sJ_{\KE,c}(M,\rho)/\Diff_0(M,\rho), \\
\sJ_{\KE,c}(M,\rho)
&:= \left\{J\in\sJ_{\INT,c}(M)\,|\,\Ric_{\rho,J}=\om_{\rho,J}/\hbar\right\}, \\
\sT_c(M) 
&:= \sJ_{\KE,c}(M)/\Diff_0(M), \\
\sJ_{\KE,c}(M)
&:= \left\{J\in\sJ_{\INT,c}(M)\,|\,J\mbox{ is K-polystable}\right\}.
\end{split}
\end{equation}
In general these spaces may be singular.
At the equivalence class of a regular element~${J\in\sJ_{\KE,c}(M,\rho)}$,
respectively~${J\in\sJ_{\KE,c}(M)}$, the tangent spaces are
\begin{equation}\label{eq:TTEICHKE}
\begin{split}
T_{[J]_\rho}\sT_c(M,\rho)  
&=
\frac{\left\{\Jhat\in\Om^{0,1}_J(M,TM)\,\big|\,
\bar\p_J\Jhat=0,\,
\Richat_\rho(J,\Jhat)\wedge\om_{\rho,J}^{\sn-1}=0
\right\}}
{\left\{\cL_{X}J\,|\,X\in\Vect(M),\,d\iota(X)\rho=0
\right\}}, \\
T_{[J]}\sT_c(M) 
&= 
\frac{\left\{\Jhat\in\Om^{0,1}_J(M,TM)\,\big|\,\bar\p_J\Jhat=0\right\}}
{\left\{\cL_{X}J\,\big|\,X\in\Vect(M)\right\}}.
\end{split}
\end{equation}
Combining the theorems of Yau, Ber\-man--Berndtsson,
and Chen--Donaldson--Sun with Moser isotopy, we find
that the natural map~${\iota_\rho:\sT_c(M,\rho)\to\sT_c(M)}$
is a bijection (and a diffeomorphism on the smooth part). 
Thus~$\sT_c(M,\rho)$ inherits the complex structure from~$\sT_c(M)$, 
and~$\sT_c(M)$ inherits the symplectic form from~$\sT_c(M,\rho)$. 

\begin{lemma}[{\bf Hodge decomposition}]\label{le:DECOMP}
Let~${J\in\sJ_{\KE,c}(M)}$, choose~$\om\in\sS_J$ 
with~${\Ric_{\om^\sn/\sn!,J}=\om/\hbar}$,
define~${\rho:=\om^\sn/\sn!}$, and 
let~${\Jhat\in\Om^{0,1}_J(M,TM)}$ with~${\bar\p_J\Jhat=0}$.   
Then there exist~${X\in\Vect(M)}$, ${F,G\in\Om^0(M)}$, 
and~${A\in\Om^{0,1}_J(M,TM)}$ such that 
\begin{equation}\label{eq:DECOMP1}
\Jhat = \cL_XJ + \cL_{v_F}J + \cL_{Jv_G}J + A,
\end{equation}
\begin{equation}\label{eq:DECOMP2}
d\iota(X)\rho = d\iota(JX)\rho = 0,\qquad
A=A^*,\qquad \bar\p_JA=0,\qquad \bar\p_J^*A=0.
\end{equation}
Thus~${\Lambda_\rho(J,A)=0}$ by Remark~\ref{rmk:LAMBDA}.
Moreover,~$X$ and~$A$ are uniquely determined by~$\Jhat$,
the four summands in~\eqref{eq:DECOMP1} 
are pairwise $L^2$ orthogonal, and
\begin{equation}\label{eq:DECOMP3}
\begin{split}
\Lambda_\rho(J,\Jhat) 
&= 
\tfrac{2}{\hbar}\iota(X)\om 
+ d\bigl(\tfrac{2}{\hbar}F-d^*dF\bigr)
- d\bigl(\tfrac{2}{\hbar}G-d^*dG\bigr)\circ J.
\end{split}
\end{equation}
\end{lemma}

\begin{proof}
The proof has five steps.

\medskip\noindent{\bf Step~1.}
{\it If~${F,G\in\Om^0(M)}$ and~${X\in\Vect(M)}$ satisfies
${d\iota(X)\rho = d\iota(JX)\rho = 0}$, then}
$$
\Lambda_\rho(J,\cL_XJ+\cL_{v_F}J+\cL_{Jv_G}J) 
= 
\tfrac{2}{\hbar}\iota(X)\om 
+ d\bigl(\tfrac{2}{\hbar}F-d^*dF\bigr)
- d\bigl(\tfrac{2}{\hbar}G-d^*dG\bigr)\circ J.
$$

\medskip\noindent
This follows from~\eqref{eq:LAMBDA3} with~${f_{v_F}=0}$ and~${f_{Jv_F}=-d^*dF}$.

\medskip\noindent{\bf Step~2.}
{\it There exist functions~${\Phi,\Psi\in\Om^0(M)}$
and vector fields~${X,Y\in\Vect(M)}$ 
such that~${d\iota(X)\rho=d\iota(JX)\rho=0}$ and}
\begin{equation}\label{eq:JHATPP}
\Lambda_\rho(J,\Jhat)
= \tfrac{2}{\hbar}\iota(Y)\om
= \tfrac{2}{\hbar}\bigl(\iota(X)\om+d\Phi-d\Psi\circ J\bigr),\qquad
Y = X + v_\Phi + Jv_\Psi.
\end{equation}

\noindent
Choose~${Y\in\Vect(M)}$ such that
$
\iota(Y)\om=\tfrac{\hbar}{2}\Lambda_\rho(J,\Jhat)
$
and then choose~${\Phi,\Psi\in\Om^0(M)}$ and~${X\in\Vect(M)}$ 
such that~${d\iota(X)\rho=d\iota(JX)\rho=0}$ and~${Y=X+v_\Phi+Jv_\Psi}$.

\medskip\noindent{\bf Step~3.}
{\it Let~${H\in\Om^0(M)}$.  Then~${d(d^*dH-\tfrac{2}{\hbar}H)=0}$
if and only if~${\cL_{v_H}J=0}$.}

\medskip\noindent
We have~${\iota(Jv_H)\rho=*dH}$ 
and~${\Lambda_\rho(J,\cL_{v_H}J)=d\bigl(\tfrac{2}{\hbar}H-d^*dH\bigr)}$
by~\eqref{eq:LAMBDA3}.  Hence
\begin{equation}\label{eq:LHJ}
\begin{split}
\tfrac{1}{2}\Norm{\cL_{v_H}J}^2
&= 
- \int_M\Lambda_\rho(J,\cL_{v_H}J)\wedge\iota(Jv_H)\rho 
= 
\int_M\bigl(d^*dH-\tfrac{2}{\hbar}H\bigr)(d^*dH)\rho
\end{split}
\end{equation}
by~\eqref{eq:LAMBDA2} and this proves Step~3.

\medskip\noindent{\bf Step~4.}
{\it Let~${H\in\Om^0(M)}$ such that~${d^*dH=\tfrac{2}{\hbar}H}$.
Then~${\int_M\Phi H\rho=\int_M\Psi H\rho=0}$.}

\medskip\noindent
By Step~3 we have~${\cL_{v_H}J=0}$ and~${\cL_{Jv_H}J=0}$.  
Hence, by~\eqref{eq:LAMBDA2} and~\eqref{eq:JHATPP},
\begin{equation*}
\begin{split}
0
&= \tfrac{\hbar}{2}\int_M\Lambda_\rho(J,\Jhat)\wedge\iota(Jv_H)\rho
= \int_Md\Phi\wedge\iota(Jv_H)\rho
= \int_M \Phi(d^*dH)\rho
= \tfrac{2}{\hbar}\int_M\Phi H\rho, \\
0
&= - \tfrac{\hbar}{2}\int_M\Lambda_\rho(J,\Jhat)\wedge\iota(v_H)\rho
= \int_Md\Psi\wedge\iota(Jv_H)\rho
= \int_M \Psi(d^*dH)\rho
= \tfrac{2}{\hbar}\int_M\Psi H\rho.
\end{split}
\end{equation*}
This proves Step~4.

\medskip\noindent{\bf Step~5.}
{\it We prove the lemma.}

\medskip\noindent
By Step~4 there exist~${F,G\in\Om^0(M)}$ 
such that~${\tfrac{2}{\hbar}F-d^*dF = \tfrac{2}{\hbar}\Phi}$, 
${\tfrac{2}{\hbar}G-d^*dG = \tfrac{2}{\hbar}\Psi}$.
Thus~${Z:=X+v_F+Jv_G}$ 
satisfies~${\Lambda_\rho(J,\cL_ZJ)=\Lambda_\rho(J,\Jhat)}$
by Step~1 and Step~2, 
and so~\eqref{eq:DECOMP1} and~\eqref{eq:DECOMP3} 
hold with~${A:=\Jhat-\cL_ZJ}$.
Let~${\omhat:=d\iota(Y)\om=\hbar\Richat_\rho(J,\Jhat)}$.
Then, since~${\bar\p_J\Jhat=0}$, it follows from~\cite[Lemma~3.6]{GPST} that
$$
\om(\Jhat u,Jv) + \om(Ju,\Jhat v)
= \omhat(u,v)-\omhat(Ju,Jv) 
= \om((\cL_YJ)u,Jv) + \om(Ju,(\cL_YJ)v)
$$
for all~${u,v\in\Vect(M)}$.  
Hence the endomorphism~${A=(\Jhat-\cL_YJ)+\cL_{Y-Z}J}$ 
of~$TM$ is symmetric.  This proves~\eqref{eq:DECOMP2} 
and Lemma~\ref{le:DECOMP}. 
\end{proof}

\begin{theorem}[{\bf Weil--Petersson symplectic form}]\label{thm:WPKE}
Let~${J\in\sJ_{\KE,c}(M)}$ and let~${\om,\rho,\Jhat_i,X_i,F_i,G_i,A_i}$ 
for~${i=1,2}$ be as in Lemma~\ref{le:DECOMP}.  Then
\begin{equation}\label{eq:TEICHSYMPKE}
\begin{split}
\Om_J^\WP\bigl(\Jhat_1,\Jhat_2\bigr)  
&:=
\tfrac{1}{2}\int_M\trace(A_1JA_2)\rho \\
&\phantom{:}=
\tfrac{1}{2}\int_M\trace(\Jhat_1J\Jhat_2)\rho
- \tfrac{2}{\hbar}\int_M\om_{\rho,J}(X_1,X_2)\rho \\
&\quad\;
+ \int_M \Bigl(\bigl(d^*dF_1-\tfrac{2}{\hbar}F_1\bigr)d^*dG_2
- \bigl(d^*dG_1-\tfrac{2}{\hbar}G_1\bigr)d^*dF_2\Bigr)\rho.
\end{split}
\end{equation}
The formula~\eqref{eq:TEICHSYMPKE} defines 
a K\"ahler form on the Teichm\"uller space~$\sT_c(M)$ 
and the mapping class group~${\Diff_c(M)/\Diff_0(M)}$ 
of isotopy classes of diffeomorphisms that preserve the 
cohomology class~$c$ acts on~$\sT_c(M)$ by K\"ahler isometries.
The pullback of~$\Om^\WP$ to~$\sT_c(M,\rho)$ 
is the $2$-form~$\Om^\WP_\rho$ given by
\begin{equation}\label{eq:WPrho}
\begin{split}
\Om_{\rho,J}^\WP\bigl(\Jhat_1,\Jhat_2\bigr) 
&=
\tfrac{1}{2}\int_M\trace(\Jhat_1J\Jhat_2)\rho  \\
&\quad
- \tfrac{\hbar}{2}\int_M\Lambda_\rho(J,\Jhat_1)\wedge\Lambda_\rho(J,\Jhat_2) 
\wedge \frac{\om_{\rho,J}^{\sn-1}}{(\sn-1)!} 
\end{split}
\end{equation}
for $J\in\sJ_{\KE,c}(M,\rho)$ and $\Jhat_i\in\Om^{0,1}_J(M,TM)$
with~${\bar\p_J\Jhat_i=0}$, ${\Richat_\rho(J,\Jhat_i)\wedge\om_{\rho,J}^{\sn-1}=0}$.
\end{theorem}

\begin{proof}
The second equality in~\eqref{eq:TEICHSYMPKE} follows from~\eqref{eq:LHJ}
and the fact that the terms~$\cL_XJ$, $\cL_{v_F}J$, $\cL_{Jv_G}J$, $A$
in Lemma~\ref{le:DECOMP} are pairwise $L^2$ orthogonal.
Also, for each~${J\in\sJ_{\KE,c}(M)}$, it follows directly
from the definition that the skew-symmetric bilinear form~$\Om^\WP_J$
on the kernel of~$\bar\p_J$ in~$\Om^{0,1}_J(M,TM)$
descends to a nondegenerate form on the quotient
space~${T_{[J]}\sT_c(M)=\ker\bar\p_J/\im\bar\p_J}$
that is compatible with the linear complex
structure~${[\Jhat]\mapsto[-J\Jhat]}$.
That the $2$-form~$\Om^\WP$ descends to~$\sT_c(M)$
and is preserved by the action of the mapping class group
follows from the naturality of the decomposition in Lemma~\ref{le:DECOMP}.
Its pullback to~$\sT_c(M,\rho)$ is given by~\eqref{eq:WPrho}, 
because~${\Richat_\rho(J,\Jhat)\wedge\om^{\sn-1}=0}$ 
implies~${d(\tfrac{2}{\hbar}G-d^*dG)=0}$ in Lemma~\ref{le:DECOMP}.   
That this pullback is closed follows from 
the equations $\p_t\om_{\rho,J} = \tfrac{\hbar}{2}d\Lambda_\rho(J,\p_tJ)$
and~${\p_s\Lambda_\rho(J,\p_tJ)-\p_t\Lambda_\rho(J,\p_sJ)
= - \tfrac{1}{2}d\trace((\p_sJ)J(\p_tJ))}$
(in~\cite[Theorem~2.7]{GPST}) for every smooth 
map~${\R^2\to\sJ_{\KE,c}(M,\rho):(s,t)\mapsto J_{s,t}}$. 
\end{proof}

The space~$\sK(M,a,\rho)$ of all  K\"ahler
pairs~$(\om,J)$ that satisfy~${[\om]=a:=2\pi\hbar c}$ and~$\om^\sn/\sn!=\rho$
is not a symplectic submanifold of $\sP(M,a,\rho)$ whenever ${\dim(M)>2}$.
In fact, using Lemma~\ref{le:DECOMP} one can show that the kernel 
of the restriction of the $2$-form~\eqref{eq:OMSE} to~$\sK(M,a,\rho)$ 
at every K\"ahler--Einstein pair~${(\om,J)\in\sK(M,a,\rho)}$ 
with~${\Ric_{\rho,J}=\om/\hbar}$
is the set of all pairs~$(\cL_X\om,\cL_XJ)$ such 
that both~$X$ and~$JX$ are divergence-free.
Nevertheless, if~${H^1(M;\R)=0}$, then
Theorem~\ref{thm:WPKE} shows that
the regular part of~$\sT_c(M,\rho)$ embeds
via~${[J]\mapsto[\om_{\rho,J},J]}$ as a symplectic submanifold
into the symplectic quotient~$\sW_\SE(M,a,\rho)$ in~\eqref{eq:WSE}. 
The condition~${H^1(M;\R)=0}$ is necessarily satisfied in the 
Fano case~${\hbar>0}$. In the case~${\hbar<0}$ the group~$\Diff_0(M)$ 
acts on~${\sJ_{\KE,c}(M)=\sJ_{\INT,c}(M)}$ with finite isotropy, the quotient 
$$
\sZ_c(M,\rho) := \sJ_{\KE,c}(M,\rho)/\Diff^\ex(M,\rho)
$$
embeds symplectically into~$\sW_\SE(M,a,\rho)$ via~${[J]\mapsto[\om_{\rho,J},J]}$,
and the space~$\sZ_c(M,\rho)$ fibers over~$\sT_c(M,\rho)$ with 
symplectic fibers.  If the action of the group~$\Diff_0(M,\rho)$ 
on~$\sJ_{\KE,c}(M,\rho)$ is free, then each fiber is isomorphic 
to~$H^{2\sn-1}(M;\R)/\Gamma_\rho$, where~$\Gamma_\rho$ is the image of the 
flux homomorphism~${\Flux_\rho:\pi_1(\Diff_0(M,\rho))\to H^{2\sn-1}(M;\R)}$.

 
\subsection{Constant scalar curvature K\"ahler metrics}
\label{subsec:TEICHCSCK}

Let~$(M,\om)$ be a closed connected symplectic $2\sn$-manifold
with the volume form~${\rho:=\om^\sn/\sn!}$ and 
denote by~$\sJ_\INT(M,\om)$ the space of all 
complex structures on~$M$ that are compatible with~$\om$.   
This space is connected for rational and ruled surfaces~\cite{AGK2},
however, in general this is an open question.
The regular part of~$\sJ_\INT(M,\om)$ is an infinite-dimensional 
K\"ahler submanifold of~$\sJ(M,\om)$ whose tangent space
at a regular element~$J$~is
$$
T_J\sJ_\INT(M,\om) 
= \left\{J\in\Om^{0,1}_J(M,TM)\,\big|\,\Jhat=\Jhat^*,\,\bar\p_J\Jhat=0\right\}.
$$
Consider the symplectic quotient
\begin{equation}\label{eq:ZCSCK}
\begin{split}
\sZ(M,\om)
&:=
\sJ_\cscK(M,\om)/\Ham(M,\om), \\
\sJ_\cscK(M,\om)
&:=
\bigl\{J\in\sJ_\INT(M,\om)\,|\,S_{\om,J}=c_\om\bigr\}.
\end{split}
\end{equation}
The regular part of this space is a complex submanifold 
of the infinite-dimensional quotient~$\sW_\csc(M,\om)$ 
in~\eqref{eq:WSCALAR} with the tangent spaces
\begin{equation}\label{eq:TZCSCK}
\begin{split}
T_{[J]_\om}\sZ(M,\om) 
= \frac{\left\{
\Jhat\in\Om^{0,1}_J(M,TM)\,\big|\,
\bar\p_J\Jhat=0,\,\Jhat=\Jhat^*,\,\Shat_\om(J,\Jhat)=0
\right\}}
{\left\{\cL_{v_H}J\,\big|\,H\in\Om^0(M)\right\}}
\end{split}
\end{equation}
for~${J\in\sJ_\cscK(M,\om)}$. So~$\sZ(M,\om)$ inherits
the K\"ahler structure of~$\sW_\csc(M,\om)$ with the 
symplectic from~\eqref{eq:OMom} and the complex 
structure~${[\Jhat]_\om\mapsto[-J(\Jhat-\cL_{v_H}J)]_\om}$, 
where~${H\in\Om^0(M)}$ 
satisfies~${\Shat_\om(J,J\Jhat-J\cL_{v_H}J)=0}$
as in Proposition~\ref{prop:HARMONIC}. 
To describe~$\sZ(M,\om)$ as a complex quotient, 
we digress briefly into GIT.

\begin{remark}[{\bf Geometric invariant theory}]\label{rmk:GIT}
Let~$(X,\om,J)$ be a closed K\"ahler manifold and let~$\rG$ be 
a compact Lie group which acts on~$X$ by K\"ahler isometries. 
Assume that the Lie algebra~${\fg:=\Lie(\rG)}$ 
is equipped with an invariant inner product and 
that the action is Hamiltonian and is generated
by an equivariant moment map~${\mu:X\to\fg}$. 
Then the complexified group~$\rG^c$ acts on~$X$ 
by holomorphic diffeomorphisms and the symplectic 
quotient~$X\dslash\rG:=\mu^{-1}(0)/\rG$ is naturally isomorphic
to the complex quotient~$X^\ps/\rG^c$ of the 
set
$
X^\ps:=\{x\in X\,|\,\rG^c(x)\cap\mu^{-1}(0)\ne\emptyset\}
$
of {\bf $\mu$-polystable} elements of~$X$ by the complexified group.
The set of $\mu$-polystable elements can be characterized in terms 
of {\bf Mumford's numerical invariants}
\begin{equation}\label{eq:WEIGHT}
w_\mu(x,\xi) := \lim_{t\to\infty}\inner{\mu(\exp(\bi t\xi)x)}{\xi}
\end{equation}
associated to~${x\in X}$ and~${0\ne\xi\in\fg}$. The 
{\bf moment-weight inequality} asserts that
\begin{equation}\label{eq:MW}
\sup_{0\ne\xi\in\fg}\frac{-w_\mu(x,\xi)}{\abs{\xi}}
\le \inf_{g\in\rG^c}\abs{\mu(gx)}
\end{equation}
and that equality holds whenever the right hand side is positive.
An immediate consequence is that an element~${x\in X}$ is 
{\bf $\mu$-semistable} (i.e.\ ${\overline{\rG^c(x)}\cap\mu^{-1}(0)\ne\emptyset}$)
if and only if~${w_\mu(x,\xi)\ge0}$ for all~${\xi\in\fg\setminus\{0\}}$.
For $\mu$-polystability the additional condition is required 
that~${w_\mu(x,\xi)=0}$ implies~${\lim_{t\to\infty}\exp(\bi t\xi)x\in\rG^c(x)}$.
This is the content of the {\bf Hilbert--Mumford criterion}.  Its proof is based on the 
study of the gradient flow of the moment map 
squared~${f:=\tfrac{1}{2}\abs{\mu}^2:X\to\R}$
and the related gradient flow of the 
{\bf Kempf--Ness function}~${\Phi_x:\rG^c/\rG\to\R}$, defined by
\begin{equation}\label{eq:KN}  
d\Phi_x(g)\ghat = -\inner{\mu(g^{-1}x)}{\Im(g^{-1}\ghat)},\qquad
\Phi_x(u) = 0
\end{equation}
for~${\ghat\in T_g\rG^c}$ and~${u\in\rG}$.
The symmetric space~$\rG^c/\rG$ is a Hadamard space and
the Kempf--Ness function is convex along geodesics.
By the {\bf Kempf--Ness Theorem}
it is bounded below if and only if~$x$ is $\mu$-semistable.
For an exposition see~\cite{GRS}. 
\end{remark}

\begin{remark}[{\bf The space of K\"ahler potentials}]\label{rmk:KPOT}
It was noted by Donaldson in his landmark paper~\cite{DON2} that much
of geometric invariant theory carries over (in part conjecturally) 
to the infinite-dimensional setting where~$X$ is replaced 
by the space~$\sJ(M,\om)$ and the compact 
Lie group~$\rG$ by~${\sG=\Ham(M,\om)}$.  
While in this situation there is no complexified group
there do exist complexified group orbits.
In the integrable case the complexified group orbit 
of~${J\in\sJ_\INT(M,\om)}$ is the space~$\sG^c(J)$
of all elements~${J'\in\sJ_\INT(M,\om)}$ that are {\bf exact isotopic to~$J$}
(i.e.\ there exists a smooth path~${[0,1]\to\sJ_\INT(M,\om):s\mapsto J_s}$  
and a smooth family of vector fields~${[0,1]\to\Vect(M,\om):s\mapsto v_s}$
such that~${J_0=J}$, ${J_1=J'}$, ${\p_sJ_s=\cL_{v_s}J_s}$, 
and~${\iota(v_s)\om=d\Phi_s - d\Psi_s\circ J_s}$
for~${\Phi_s,\Psi_s\in\Om^0(M)}$).  In this situation 
the role of the symmetric space~$\rG^c/\rG$ is played by the 
{\bf space of K\"ahler potentials}
\begin{equation}\label{eq:KPOT}
\sH_J := \left\{h\in\Om^0(M)\,\bigg|\,
\begin{array}{l}
\mbox{The $2$-form }\om_h:=\om+\tfrac{1}{2}d(dh\circ J) = \om - \bi\p\bar\p h \\
\mbox{satisfies }\om_h(\xhat,J\xhat)>0\;\forall\;0\ne\xhat\in TM
\end{array}\right\}.
\end{equation}
This space has been studied by Calabi, Chen~\cite{CALABI1,CC,C,C2},
Mabuchi~\cite{MABUCHI1,MABUCHI2}, Semmes~\cite{S}, 
Donaldson~\cite{DON2} and others.  
It is an infinite-dimensional symmetric space of nonpositive 
sectional curvature with the Mabuchi metric 
\begin{equation}\label{eq:KPOTmetric}
\inner{\hhat_1}{\hhat_2}_h 
:= \int_M\hhat_1\hhat_2\frac{\om_h^\sn}{\sn!}
\end{equation}
and geodesics are the solutions~${t\mapsto h_t}$
of the {\bf Monge--Amp\`ere equation}
\begin{equation}\label{eq:KPOTgeo}
\p_t\p_th_t + \tfrac{1}{2}\Abs{d\p_th_t}^2_{h_t} = 0.
\end{equation}
In~\cite{C} Chen proved that any two elements 
of~$\sH_J$ can be joined by a weak~$C^{1,1}$ geodesic.  
As noted by Donaldson~\cite{DON1,DON2}, 
the analogues of Mumford's numerical invariants in this setting
are the {\bf Futaki invariants}~\cite{FUTAKI}, the analogue of the 
Kempf--Ness function is the {\bf Mabuchi functional}~\cite{MABUCHI1}
${\cM_J:\sH_J\to\R}$ defined by 
\begin{equation}\label{eq:MABUCHI}
d\cM_J(h)\hhat = \int_M\bigl(S_{\om_h,J}-c_\om\bigr)\hhat\frac{\om_h^\sn}{\sn!},
\qquad \cM_J(0)=0, 
\end{equation}
and the analogue of the gradient flow of the moment map squared
is the Calabi flow.  After earlier results by Donaldson~\cite{DONcsc1,DONcsc2},
Mabuchi~\cite{MABUCHI3}, and Chen--Tian~\cite{CT} it was shown by 
Berman--Berndtsson~\cite{BB} that the Mabuchi functional 
is convex along weak geodesics, 
that every K\"ahler potential~${h\in\sH_J}$ with  
constant scalar curvature~${S_{\om_h,J}=c_\om}$ 
minimizes the Mabuchi functional, 
and that constant scalar curvature K\"ahler metrics 
are unique up to holomorphic diffeomorphism, 
i.e.\ if two K\"ahler potentials~${h,h'\in\sH_J}$ 
have constant scalar curvature~${S_{\om_h,J}=S_{\om_{h'},J}=c_\om}$, 
then there exists a holomorphic diffeomorphism~${\psi\in\Aut_0(M,J)}$
such that~${\om_{h'}=\psi^*\om_h}$.  

In the present setting the analogue of the Hilbert--Mumford criterion 
is the Yau--Tian--Donaldson conjecture~\cite{DON4,TIAN0,YAU2}
which relates the existence of a constant scalar curvature 
K\"ahler potential to K-polystability.  
For Fano manifolds it was confirmed by 
Chen--Donaldson--Sun~\cite{CDS0,CDS123},
while in general it is an open question.
\end{remark}

Remark~\ref{rmk:KPOT} shows that, according to the YTD conjecture, 
the space~$\sZ(M,\om)$ can be expressed as the 
quotient~${\sZ_K(M,\om) := \sJ_K(M,\om)/\!\sim}$,
where~$\sJ_K(M,\om)$ is the space of K-polystable complex 
structures that are compatible with~$\om$, and the equivalence 
relation is exact isotopy as in Remark~\ref{rmk:KPOT}. 
The formal (Zariski type) tangent space of~$\sZ_K(M,\om)$ 
at the equivalence class of an element~${J\in\sJ_K(M,\om)}$ 
is the quotient~${T_{[J]}\sZ_K(M,\om)
=T_J\sJ_\INT(M,\om)/\{\cL_{v_F+Jv_G}J\,|\,F,G\in\Om^0(M)\}}$
and the complex structure is~${[\Jhat]\mapsto[-J\Jhat]}$.

It is also of interest to consider the {\bf Teichm\"uller space of 
constant scalar curvature K\"ahler metrics}, defined by 
\begin{equation}\label{eq:TEICHCSCK}
\begin{split}
\sT(M,\om)  :=  \sJ_\cscK(M,\om)/\Symp_0(M,\om).
\end{split}
\end{equation}
If~${H^1(M;\R)=0}$, then~${\Ham(M,\om)=\Symp_0(M,\om)}$
and hence the regular part of the Teich\-m\"uller space~${\sT(M,\om)=\sZ(M,\om)}$ is K\"ahler.
The two spaces also agree in the Calabi--Yau case, 
where~$\Symp_0(M,\om)/\Ham(M,\om)$ 
acts trivially on~$\sZ(M,\om)$.
In the K\"ahler--Einstein case with~${2\pi\hbar c_1^\R(\om)=[\om]}$
and~${\hbar<0}$ the group~$\Symp_0(M,\om)$ acts on~$\sJ_\cscK(M,\om)$
with finite isotropy, the Teich\-m\"uller space~$\sT(M,\om)$ carries 
a K\"ahler form given by~\eqref{eq:WPrho} with~${\om_{\rho,J}=\om}$,
and~$\sZ(M,\om)$ fibers over~$\sT(M,\om)$ with symplectic fibers.  
If the action of~$\Symp_0(M,\om)$ on~$\sJ_\cscK(M,\om)$ is free, 
then each fiber is isomorphic to the space~${H^1(M;\R)/\Gamma_\om}$,
where~$\Gamma_\om$ is the image of the flux
homomorphism~${\Flux_\om:\pi_1(\Symp_0(M,\om))\to H^1(M;\R)}$.
 
\begin{remark}\label{rmk:TEICHomrho}
Fix a symplectic form~$\om$ on~$M$ that admits a compatible 
complex structure~$J$ with~${S_{\om,J}=c_\om}$ 
and~${\Aut(M,J)\cap\Diff_0(M)=\Aut_0(M,J)}$.  Define
$$
\rho:= \om^\sn/\sn!,\qquad a=[\om]\in H^2(M;\R),\qquad
c := c_1^\R(\om)\in H^2(M;\R).
$$
Assume the Calabi--Yau case~${c=0}$
and consider the {\bf polarized Teichm\"uller space}
\begin{equation}\label{eq:TEICHCYa}
\begin{split}
\sT_{0,a}(M,\rho) 
:= \left\{J\in\sJ_{\INT,0}(M)\,|\,\Ric_{\rho,J}=0,\,a\in\cK_J\right\}/\Diff_0(M,\rho)
\end{split}
\end{equation}
of all isotopy classes of Ricci-flat complex structures
that contain the cohomology class~$a$ in their K\"ahler cone. 
This space is Hausdorff and the Weil--Petersson metric 
on~$\sT_{0,a}(M,\rho)$ is K\"ahler.
Moreover, there is a natural holomorphic map
$$
\iota_\om:\sT(M,\om)\to\sT_{0,a}(M,\rho)
$$
which pulls back the Weil--Petersson symplectic form 
on~$\sT_{0,a}(M,\rho)$ in Theorem~\ref{thm:WPCY}
to the symplectic form~\eqref{eq:OMom} on~$\sT(M,\om)$. 
The map~$\iota_\om$ need not be injective or surjective.  
It is injective if and only 
if~${\Symp(M,\om)\cap\Diff_0(M)=\Symp_0(M,\om)}$.
In~\cite{SEIDEL3} Seidel found many
examples of symplectic four-manifolds that admit 
symplectomorphisms that are smoothly, but not symplectically, 
isotopic to the identity, including K3-surfaces with embedded 
Lagrangian spheres. 
The map~$\iota_\om$ is surjective if and only if the space~$\sS_{0,a}$
of symplectic forms in the class~$a$ with real first Chern class zero
that admit compatible complex structures is connected. 
By a theorem of Hajduk--Tralle~\cite{HT} the space~$\sS_{0,a}$ 
is disconnected for the $8k$-torus with~${k\ge1}$.

Now assume the K\"ahler--Einstein case~${2\pi\hbar c = a}$
for some nonzero real number~$\hbar$.  Then there is again 
a natural holomorphic map
$$
\iota_\om:\sT(M,\om)\to\sT_c(M,\rho)
$$
which pulls back the Weil--Petersson symplectic form 
on~$\sT_c(M,\rho)$ in Theorem~\ref{thm:WPKE}
to the symplectic form~\eqref{eq:WPrho} on~$\sT(M,\om)$. 
As before, this map need not be injective or surjective.
Seidel's examples in dimension four include as Fano mani\-folds 
the $k$-fold blowup of the projective plane with~${5\le k\le8}$
and many K\"ahler surfaces of general type, 
so in these cases the map~$\iota_\om$ is not injective.
It is surjective if and only if the space~$\sS_{c,a}$ 
of symplectic forms in the class~$a$ with first Chern class~$c$ 
that admit compatible K-polystable complex structures is connected.
By a theorem of Randal-Williams~\cite{ORL}  
every complete intersection~$M$ with~${\dim(M)=16k\ge16}$ 
and~${\dim(H^{8k}(M;\R))\ge6}$ admits a diffeomorphism~$\phi$
that acts as the identity on homology
such that~$\om$ is not isotopic to~$\phi^*\om$ 
for every symplectic form~$\om$.  
This includes  K\"ahler--Einstein examples, 
and in these cases~$\sS_{c,a}$ is disconnected.
\end{remark}


\section{Fano manifolds}\label{sec:FANO}

This section explains how the symplectic form introduced by
Donaldson~\cite{DON5} on the space of Fano complex structures
fits into the present setup.   We begin by giving another proof of nondegeneracy,
and then discuss Berndtsson convexity for the Ding functional
and the Donaldson--K\"ahler--Ricci flow.


\subsection{The Donaldson symplectic form}\label{subsec:DON}
Let~$(M,\om)$ be a closed connected symplectic $2\sn$-manifold 
that satisfies the Fano condition
\begin{equation}\label{eq:FANO}
2\pi c_1^\R(\om) = [\om] \in H^2(M;\R).
\end{equation}
As in Subsection~\ref{subsec:SCALAR} we denote by~$v_F$ 
the Hamiltonian vector field of~$F$ and 
by~${\{F,G\}}$ the Poisson bracket 
of~${F,G\in\Om^0(M)}$.  Let~$\sJ_\INT(M,\om)$ be the space 
of $\om$-compatible complex structures on~$M$. 
If this space is nonempty, then~$M$ is simply connected.
Throughout we will ignore all regularity issues 
and treat~$\sJ_\INT(M,\om)$ as a submanifold 
of~$\sJ(M,\om)$ whose tangent space 
at~${J\in\sJ_\INT(M,\om)}$ is 
$$
T_J\sJ_\INT(M,\om) 
= \left\{\Jhat\in\Om^{0,1}_J(M,TM)\,\big|\,
\bar\p_J\Jhat=0,\,\Jhat=\Jhat^*\right\}.
$$
This is a complex subspace of~$T_J\sJ(M,\om)$
and so inherits the symplectic form~\eqref{eq:OMom}
from the ambient K\"ahler manifold~$\sJ(M,\om)$. 
In~\cite{DON5} Donaldson introduced another symplectic form 
on~$\sJ_\INT(M,\om)$ which we explain next. 

It follows from~\eqref{eq:RICf} that,
for every complex structure~${J\in\sJ_\INT(M,\om)}$, 
there exists a unique positive volume 
form~${\rho_J\in\Om^{2\sn}(M)}$ that satisfies
\begin{equation}\label{eq:rhoJFANO}
\Ric_{\rho_J,J}=\om,\qquad\int_M\rho_J = 1.
\end{equation}
Moreover, ~$\Richat_{\rho_J}(J,\Jhat)$ is an exact $(1,1)$-form 
for every~${\Jhat\in T_J\sJ_\INT(M,\om)}$ by~\eqref{eq:rhoJFANO} 
and~\cite[Lemma~3.6]{GPST}.  Thus, by the $\p\bar\p$-lemma,
for every~${J\in\sJ_\INT(M,\om)}$  and every~${\Jhat\in T_J\sJ_\INT(M,\om)}$, 
there exist unique functions~${f,g\in\Om^0(M)}$ that satisfy
\begin{equation}\label{eq:LAMBDAfgFANO}
\Lambda_{\rho_J}(J,\Jhat) = -df\circ J + dg,\qquad
\int_Mf\rho_J = \int_Mg\rho_J=0.
\end{equation}
The {\bf Donaldson symplectic form}
on~$\sJ_\INT(M,\om)$ is defined by
\begin{equation}\label{eq:DONSYMP}
\Om_J^D(\Jhat_1,\Jhat_2) 
:= \int_M\Bigl(
\tfrac{1}{2}\trace\bigl(\Jhat_1J\Jhat_2\bigr)
- f_1g_2 + g_1f_2
\Bigr)\rho_J
\end{equation}
for~${J\in\sJ_\INT(M,\om)}$ and~${\Jhat_i\in T_J\sJ_\INT(M,\om)}$, 
where~${\rho_J,f_i,g_i}$ 
are as in~\eqref{eq:rhoJFANO} and~\eqref{eq:LAMBDAfgFANO}.
The fact that the $2$-form~\eqref{eq:DONSYMP} is nondegenerate 
is far from trivial and is one of the main results in~\cite{DON5}.   Indeed, 
as noted by Donaldson, nondegeneracy can be viewed as a reformulation 
of Berndtsson's convexity theorem~\cite{BB1,BB2} for the Ding 
functional~\cite{DING} on the space of K\"ahler potentials 
(see Subsection~\ref{subsec:DING} below). 

The symplectic form~$\Om^D$ in~\eqref{eq:DONSYMP}
is given by essentially the same formula as the Weil--Petersson symplectic 
form~$\Om^\WP$ on~$\sT_0(M)$ in~\eqref{eq:OMWP}.  In contrast to the 
Calabi--Yau case, where the lifted $2$-form on~$\sJ_{\INT,0}(M)$ 
has the kernel~$\im\bar\p_J$ on each tangent space~$T_J\sJ_{\INT,0}(M)$, 
the $2$-form~$\Om^D$ is nondegenerate in the Fano case.

The definition of the symplectic form in Donaldson's paper~\cite{DON5} 
uses the existence of a holomorphic $\sn$-form with values in a
suitable holomorphic line bundle to define the volume form 
denoted by~$\rho_J$ in~\eqref{eq:rhoJFANO}.
That the $2$-form~\eqref{eq:DONSYMP} agrees 
with the symplectic form in~\cite{DON5} (up to a factor~$1/4$)
then follows from the discussion in~\cite[Appendix~D]{GPST}.

\begin{theorem}[{\bf Donaldson~\cite{DON5}}]\label{thm:DON}
$\Om^D$ is a $\Symp(M,\om)$-in\-variant symplectic form on~${\sJ_\INT(M,\om)}$
and is compatible with the complex structure~${\Jhat\mapsto-J\Jhat}$.
The action of~$\Ham(M,\om)$ on~$\sJ_\INT(M,\om)$
is Hamiltonian and is generated by the $\Symp(M,\om)$-equi\-variant 
moment map
$
\mu:\sJ_\INT(M,\om)\to(\Om^0(M)/\R)^*, 
$
given by 
\begin{equation}\label{eq:DONMOMENT1}
\inner{\mu(J)}{H} := 2\int_MH\left(\frac{1}{V}\frac{\om^\sn}{\sn!}-\rho_J\right),\qquad
V:=\int_M\frac{\om^\sn}{\sn!},
\end{equation}
for~${J\in\sJ_\INT(M,\om)}$ and~${H\in\Om^0(M)}$,
where~${\rho_J\in\Om^{2\sn}(M)}$ 
is as in~\eqref{eq:rhoJFANO}.   Thus
\begin{equation}\label{eq:DONMOMENT2}
\Om^D_J(\Jhat,\cL_{v_H}J) 
= - 2\int_MHf \rho_J
= \inner{d\mu(J)\Jhat}{H}
\end{equation}
for~${\Jhat\in T_J\sJ_\INT(M,\om)}$ and~${H\in\Om^0(M)}$, 
where~$f$ is as in~\eqref{eq:LAMBDAfgFANO}.
\end{theorem}

\begin{remark}[{\bf Weil--Petersson symplectic form}]\label{rmk:WP}
The zero set of the moment map in Theorem~\ref{thm:DON} 
is the space
$$
\sJ_\KE(M,\om)
:=
\left\{J\in\sJ_\INT(M,\om)\,|\,\Ric_{\om^\sn/\sn!,J}=\om\right\}
=
\sJ_\cscK(M,\om)
$$
of K\"ahler--Einstein complex structures compatible with~$\om$.
Since~$M$ is simply connected, the  
quotient~${\sT_\KE(M,\om):=\sJ_\KE(M,\om)/\Ham(M,\om)=\sT(M,\om)}$ 
is the Teich\-m\"uller space in~\eqref{eq:TEICHCSCK} and
the symplectic form on this space induced 
by~\eqref{eq:DONSYMP} is $1/V$ times the Weil--Petersson
symplectic form induced by~\eqref{eq:OMom}.
For the relation to the Teichm\"uller space~${\sT_c(M,\rho)\cong\sT_c(M)}$
in Theorem~\ref{thm:WPKE} see Remark~\ref{rmk:TEICHomrho}.
\end{remark}

Below we give a proof of nondegeneracy of~\eqref{eq:DONSYMP} 
which amounts to translating the argument in~\cite{DON5} 
into our notation. The heart of the proof is Lemma~\ref{le:BERNDTSSON}.

\begin{definition}\label{def:THETAJ}
Fix a complex structure~${J\in\sJ_\INT(M,\om)}$
and let~${\rho:=\om^\sn/\sn!}$.
The {\bf K\"ahler--Ricci potential} of $J$ is the 
function $\Theta_{\om,J}:=\Theta_J:M\to(0,\infty)$ defined by~${\Theta_J := \rho_J/\rho}$.
Hence~${\Ric_{\Theta_J\rho,J}=\om}$ by~\eqref{eq:rhoJFANO}
and so by~\eqref{eq:RICf}, we have
\begin{equation}\label{eq:THETAJ}
\tfrac{1}{2}d(d\log(\Theta_J)\circ J) = \om-\Ric_{\rho,J},\qquad
\int_M\Theta_J\rho = 1.
\end{equation}
Thus~${\Theta_J=1/V}$  if and only 
if~$(M,\om,J)$ is a K\"ahler--Einstein manifold.
Denote by~${d^*d:\Om^0(M)\to\Om^0(M)}$ the Laplace--Beltrami operator 
of the Riemannian metric~${\inner{\cdot}{\cdot}:=\om(\cdot,J\cdot)}$ 
and define the linear operators~${\bL,\bB:\Om^0(M)\to\Om^0(M)}$ by
\begin{equation}\label{eq:LB}
\bL F := d^*dF-\frac{\inner{v_{\Theta_J}}{v_F}}{\Theta_J},\qquad
\bB F := \frac{\{\Theta_J,F\}}{\Theta_J}
\end{equation}
for~${F\in\Om^0(M)}$. Thus~$\bL$ is a self-adjoint Fredholm operator 
and~$\bB$ is skew-adjoint with respect to the $L^2$ inner 
product~${\inner{F}{G}_J:=\int_MFG\rho_J}$ on~$\Om^0(M)$.
\end{definition}

\begin{lemma}[{\bf K\"ahler--Ricci potential}]\label{le:THETAHAT}
Choose elements $J\in\sJ_\INT(M,\om)$ and~${\Jhat\in T_J\sJ_\INT(M,\om)}$
and let~$f$ and~$\Theta_J$ be as in~\eqref{eq:LAMBDAfgFANO}
and~\eqref{eq:THETAJ}.  Then
\begin{equation}\label{eq:THETAHAT}
\Thetahat_\om(J,\Jhat) 
:= \left.\tfrac{\p}{\p t}\right|_{t=0}\Theta_{J_t} 
= f\Theta_J
\end{equation}
for every smooth path~${\R\to\sJ_\INT(M,\om):t\mapsto J_t}$
with~${J_0=J}$ and~${\left.\tfrac{\p}{\p t}\right|_{t=0}J_t=\Jhat}$.
\end{lemma}

\begin{proof}
By Proposition~\ref{prop:LAMBDA} and~\eqref{eq:RICf}
the derivative of any path~${t\mapsto\Ric_{\rho_t,J_t}}$ is given 
by~${\p_t\Ric_{\rho_t,J_t}=\Richat_{\rho_t}(J_t,\p_tJ_t)
+\tfrac{1}{2}d(d(\p_t\rho_t/\rho_t)\circ J_t)}$.
In the case at hand with~${J_t\in\sJ_\INT(M,\om)}$ 
and~${\rho_t=\rho_{J_t}=\Theta_{J_t}\om^\sn/\sn!}$ 
this yields the equation
$$
0 = \Richat_{\rho_J}(J,\Jhat)
+ \tfrac{1}{2}d\bigl(d\bigl(\Thetahat_\om(J,\Jhat)/\Theta_J\bigr)\circ J\bigr)
= \tfrac{1}{2}d\bigl(d\bigl(\Thetahat_\om(J,\Jhat)/\Theta_J-f\bigr)\circ J\bigr).
$$
Since~$\Thetahat_\om(J,\Jhat)/\Theta_J-f$ has mean value zero 
for~$\rho_J$, this proves the lemma.
\end{proof}

\begin{lemma}[{\bf Holomorphic vector fields}]\label{le:LB}
Let~${J\in\sJ_\INT(M,\om)}$, choose the 
quadruple~${\rho_J,\Theta_J,\bL,\bB}$ as in Definition~\ref{def:THETAJ},
and choose functions~${F,G\in\Om^0(M)}$ 
such that~${\int_MF\rho_J=\int_MG\rho_J=0}$.
Then the following holds.

\smallskip\noindent{\bf (i)}
${F=G=0}$ if and only if~${\bL F + \bB G = 0}$
and~${\bL G - \bB F = 0}$.

\smallskip\noindent{\bf (ii)}
${\Lambda_{\rho_J}(\\J,\cL_{v_F+Jv_G}J)
= -d(2G-\bL G + \bB F)\circ J + d(2F-\bL F-\bB G)}$.

\smallskip\noindent{\bf (iii)}
${\cL_{v_F+Jv_G}J=0}$ if and only if~${\bL F + \bB G = 2F}$
and~${\bL G - \bB F = 2G}$.
\end{lemma}

\begin{proof}
Throughout the proof we use the notation
$$
\bigl\|F\bigr\| := \sqrt{\int_MF^2\rho_J},\qquad
\bigl\|v\bigr\| := \sqrt{\int_M\om(v,Jv)\rho_J},\qquad
\bigl\|\Jhat\bigr\| := \sqrt{\int_M\trace(\Jhat^2)\rho_J}
$$
for~${F\in\Om^0(M)}$, ${v\in\Vect(M)}$, and~${\Jhat\in\Om^{0,1}_J(M,TM)}$
with~${\Jhat=\Jhat^*}$.

To prove part~(i), we observe that
\begin{equation}\label{eq:LBFG}
\begin{split}
\int_M(\bL F)G\rho_J
&= 
\int_M(dF\circ J)\wedge dG\wedge\Theta_J\frac{\om^{\sn-1}}{(\sn-1)!}
=
\int_M\om(v_F,Jv_G)\rho_J, 
\\
\int_M(\bB F)G\rho_J
&= 
\int_M dF\wedge dG\wedge\Theta_J\frac{\om^{\sn-1}}{(\sn-1)!}
=
\int_M\om(v_F,v_G)\rho_J
\end{split}
\end{equation}
for all~${F,G\in\Om^0(M)}$, and hence
\begin{equation}\label{eq:vFG}
\Norm{v_F+Jv_G}^2 
= \int_M\Bigl(F(\bL F + \bB G) + G(\bL G - \bB F)\Bigr)\rho_J.
\end{equation}
Now let~${F,G\in\Om^0(M)}$ 
such that~${\bL F + \bB G=\bL G - \bB F=0}$.
Then~${v_F+Jv_G=0}$ by~\eqref{eq:vFG} and
hence~${\int_M\om(v_F,Jv_F)\om^\sn/\sn!=\int_M\om(v_F,v_G)\om^\sn/\sn!=0}$.
Thus~$F$ and~$G$ are constant and this proves~(i).
Part~(ii) follows from~\eqref{eq:LAMBDA3}, \eqref{eq:rhoJFANO}, 
and 
\begin{equation}\label{eq:THETAJFG}
d\iota(v_F)\rho_J = (\bB F)\rho_J,\qquad
d\iota(Jv_G)\rho_J =  - (\bL G)\rho_J.
\end{equation}
Moreover, 
$
\tfrac{1}{2}\Norm{\cL_{v_F+Jv_G}J}^2
= -\int_M\Lambda_{\rho_J}(J,\cL_{v_F+Jv_G}J)\wedge\iota(J(v_F+Jv_G))\rho_J
$
by~\eqref{eq:LAMBDA2}, and hence~(iii) follows from~(ii).
\end{proof}

\begin{lemma}[{\bf Decomposition Lemma}]\label{le:DEC}
Let~${J\in\sJ_\INT(M,\om)}$,
let~$\rho_J$ be as in~\eqref{eq:rhoJFANO},
and let~${\Jhat\in\Om^{0,1}_J(M,TM)}$ 
such that~${\bar\p_J\Jhat=0}$ and~${\Jhat=\Jhat^*}$
with respect to the metric~${\om(\cdot,J\cdot)}$.
Then there exist~${F,G\in\Om^0(M)}$
and~${A\in\Om^{0,1}_J(M,TM)}$ such that
\begin{equation}\label{eq:DEC1}
\Jhat = \cL_{v_F}J + \cL_{Jv_G}J + A,
\end{equation}
and
\begin{equation}\label{eq:DEC2}
\begin{split}
\int_MF\rho_J=\int_MG\rho_J=0,\qquad
A=A^*, \qquad \bar\p_JA=0,\qquad 
\Lambda_{\rho_J}(J,A)=0.
\end{split}
\end{equation}
Moreover, $A$ and~${\cL_{v_F+Jv_G}J}$ are $L^2$ orthogonal
If~$\Jhat$ satisfies~\eqref{eq:DEC1} and~\eqref{eq:DEC2}, 
then~$\Lambda_{\rho_J}(J,\Jhat)$ is given by~\eqref{eq:LAMBDAfgFANO}
with 
\begin{equation}\label{eq:fgFG}
f=2G-\bL G+\bB F,\qquad
g=2F-\bL F-\bB G,
\end{equation}
and~$\Jhat$ satisfies the equation
\begin{equation}\label{eq:FGA}
\begin{split}
&\int_M\Bigl(\tfrac{1}{2}\trace\bigl(\Jhat^2\bigr) - f^2 - g^2\Bigr)\rho_J  \\
&= 
\int_M\tfrac{1}{2}\trace\bigl(A^2\bigr)\rho_J
+
2\int_M\Bigl(
\Abs{v_F+Jv_G}^2 - 2\bigl(F^2+G^2\bigr)
\Bigr)\rho_J.
\end{split}
\end{equation}
\end{lemma}

\begin{proof}
Let~${f,g\in\Om^0(M)}$ be as in~\eqref{eq:LAMBDAfgFANO}
We prove first that all~${F,G\in\Om^0(M)}$ satisfies the identity
\begin{equation}\label{eq:JHATFG}
\tfrac{1}{2}\int_M\trace(\Jhat\cL_{v_F+Jv_G}J)\rho_J 
= 
-\int_M \Bigl(g(\bL F + \bB G) + f(\bL G-\bB F)\Bigr)\rho_J.
\end{equation}
To see this, abbreviate~${v:=v_F+Jv_G}$.  
Then we have~${d\iota(v)\rho_J = (-\bL G + \bB F)\rho_J}$
and~${d\iota(Jv)\rho_J = (-\bL F - \bB G)\rho_J}$ 
by~\eqref{eq:THETAJFG}. Hence, by~\eqref{eq:LAMBDA2} 
and~\eqref{eq:LAMBDAfgFANO},
\begin{equation*}
\begin{split}
\tfrac{1}{2}\int_M\trace\bigl(\Jhat\cL_vJ\bigr)\rho_J 
&= 
\int_M\Bigl(df\circ J-dg\Bigr)\wedge\iota(Jv)\rho_J \\
&=
\int_M\Bigl(fd\iota(v)\rho_J + gd\iota(Jv)\rho_J\Bigr) \\
&=
\int_M\Bigl(f(- \bL G + \bB F) + g(- \bL F - \bB G)\Bigr)\rho_J.
\end{split}
\end{equation*}
This proves~\eqref{eq:JHATFG}.

Now choose functions~${F,G\in\Om^0(M)}$ such that
$$
\bL F+\bB G=2F,\qquad
\bL G-\bB F=2G.  
$$
Then~${\cL_{v_F+Jv_G}J=0}$
by part~(iii) of Lemma~\ref{le:LB} and hence~${\int_M(gF+fG)\rho_J=0}$
by equation~\eqref{eq:JHATFG}.  Thus the pair~$(g,f)$ 
is $L^2$ orthogonal to the kernel of the self-adjoint 
Fredholm operator
$
(F,G)\mapsto(2F-\bL F-\bB G,2G-\bL G+\bB F)
$ 
and so belongs to its image.  Hence there exist
smooth functions~${F,G\in\Om^0(M)}$ 
such that
$$
2F - \bL F - \bB G = g,\qquad
2G - \bL G + \bB F = f.
$$
By part~(ii) of Lemma~\ref{le:LB} this implies
$
\Lambda_{\rho_J}(J,\Jhat-\cL_{v_F+Jv_G}J)=0.
$
Since~$\cL_{v_F}J$ and~${\cL_{Jv_G}J=J\cL_{v_G}J}$
are symmetric, by~\cite[Lemma~3.7]{GPST}, this proves~\eqref{eq:DEC2}.  

Now assume that~${F,G,A}$ have been found such that~$\Jhat$
satisfies~\eqref{eq:DEC1} and~\eqref{eq:DEC2}. 
Then~\eqref{eq:fgFG} follows directly from part~(ii) 
of Lemma~\ref{le:LB}. Moreover, 
$$
\int_M\trace(A\cL_{v_F+Jv_G}J)\rho_J=0
$$
by~\eqref{eq:LAMBDA2} and~\eqref{eq:DEC2}.  
Hence, by~\eqref{eq:JHATFG} we have
\begin{equation*}
\begin{split}
&\tfrac{1}{2}\int_M\trace\bigl(\Jhat^2\bigr)\rho_J 
- \tfrac{1}{2}\int_M\trace\bigl(A^2\bigr)\rho_J  \\
&= 
\tfrac{1}{2}\int_M\trace\bigl(\Jhat(\cL_{v_F+Jv_G}J)\bigr)\rho_J \\
&=
\int_M\Bigl(f(- \bL G + \bB F) + g(- \bL F - \bB G)\Bigr)\rho_J \\
&=
\int_M\Bigl(f(f-2G) + g(g-2F)\Bigr)\rho_J \\
&=
\int_M\Bigl(f^2+g^2\Bigr)\rho_J 
+ 2\int_M\Bigl(F(\bL F+\bB G-2F) + G(\bL G-\bB F -2G)\Bigr)\rho_J \\
&=
\int_M\left(f^2+g^2\right)\rho_J
+ 2\int_M\Bigl(
\Abs{v_F+Jv_G}^2-2\bigl(F^2+G^2\bigr)
\Bigr)\rho_J.
\end{split}
\end{equation*}
Here the last step uses~\eqref{eq:vFG}.
This proves~\eqref{eq:FGA}.
\end{proof}

\begin{lemma}[{\bf Berndtsson Inequality}]\label{le:BERNDTSSON}
Let~${\sJ\in\sJ_\INT(M,\om)}$, 
choose~$\rho_J$ as in~\eqref{eq:rhoJFANO}, 
and let~${F,G\in\Om^0(M)}$ 
such that~${\int_MF\rho_J=\int_MG\rho_J=0}$. Then
\begin{equation}\label{eq:BERNDTSSON}
\int_M\Abs{v_F+Jv_G}^2\rho_J
\ge 2\int_M\bigl(F^2+G^2\bigr)\rho_J,
\end{equation}
and equality holds in~\eqref{eq:BERNDTSSON} if and only if~${\cL_{v_F+Jv_G}J=0}$.
If~${\cL_{v_F+Jv_G}J=0}$, then every pair of functions~${\Fhat,\Ghat\in\Om^0(M)}$ 
with~${\int_M\Fhat\rho_J=\int_M\Ghat\rho_J=0}$ satisfies
\begin{equation}\label{eq:FGFG}
\int_M\INNER{v_F+Jv_G}{v_\Fhat+Jv_\Ghat}
= 2\int_M\Bigl(F\Fhat+G\Ghat\Bigr)\rho_J.
\end{equation}
\end{lemma}

\begin{proof} 
Since~$F$ and~$G$ have mean value zero, it follows from
part~(i) of Lemma~\ref{le:LB} that there exists a unique
pair of functions~${\Phi,\Psi\in\Om^0(M)}$ such that
\begin{equation}\label{eq:LBFGPP}
\bL\Phi+\bB\Psi = F,\qquad
\bL\Psi-\bB\Phi = G,\qquad
\int_M\Phi\rho_J = \int_M\Psi\rho_J =0.
\end{equation}
Continue the notation in the proof of Lemma~\ref{le:LB} and define
\begin{equation}\label{eq:uv}
u:=v_\Phi+Jv_\Psi,\qquad v:=v_F+Jv_G.
\end{equation}
Then Lemma~\ref{le:DEC} 
with~${\Jhat=\cL_uJ}$, ${f=2\Psi-G}$, ${g=2\Phi-F}$,~${A=0}$ yields
\begin{equation}\label{eq:uFG2}
\begin{split}
\tfrac{1}{2}\Norm{\cL_uJ}^2
&=
\int_M\Bigl(
2\Abs{u}^2 - 4\bigl(\Phi^2+\Psi^2\bigr)
+ \bigl(2\Phi-F\bigr)^2
+ \bigl(2\Psi-G\bigr)^2
\Bigr)\rho_J \\
&=
\int_M\Bigl(
2\Abs{u}^2 + F^2 + G^2
- 4\Phi F - 4\Psi G
\Bigr)\rho_J \\
&=
\int_M\Bigl(F^2 + G^2\Bigr)\rho_J
- 2\Norm{u}^2.
\end{split}
\end{equation}
The last step uses the formula~${\norm{u}^2 = \int_M(\Phi F+\Psi G)\rho_J}$ 
in~\eqref{eq:vFG}.  Now, for~${\lambda\in\R}$,
\begin{equation*}
\begin{split}
&\Norm{v}^2 - \Norm{v-\lambda u}^2
=
2\lambda\int_M\om(u,Jv)\rho_J - \lambda^2\Norm{u}^2 \\
&=
2\lambda\int_M\Bigl(
\om(v_\Phi,Jv_F) + \om(v_\Psi,v_F) + \om(v_\Psi,Jv_G) - \om(v_\Phi,v_G)
\Bigr)\rho_J 
- \lambda^2\Norm{u}^2\\
&=
2\lambda\int_M\Bigl((\bL\Phi+\bB\Psi)F 
+ (\bL\Psi-\bB\Phi)G\Bigr)\rho_J 
- \lambda^2\Norm{u}^2\\
&=
2\lambda\int_M\bigl(F^2+G^2\bigr)\rho_J  -  \lambda^2\Norm{u}^2 \\
&= 
\left(2\lambda-\frac{\lambda^2}{2}\right)\int_M\bigl(F^2+G^2\bigr)\rho_J
+ \frac{\lambda^2}{4}\Norm{\cL_uJ}^2.
\end{split}
\end{equation*}
Here the second equality follows from~\eqref{eq:uv}, the third from~\eqref{eq:LBFG}, 
the fourth from~\eqref{eq:LBFGPP}, and the last from~\eqref{eq:uFG2}.
With~${\lambda=2}$ this yields
\begin{equation}\label{eq:uvFG}
\Norm{v}^2-2\int_M\bigl(F^2+G^2\bigr)\rho_J
= \Norm{\cL_uJ}^2 + \Norm{v-2u}^2
\ge 0. 
\end{equation}
This proves~\eqref{eq:BERNDTSSON}.
Moreover, equality in~\eqref{eq:BERNDTSSON} implies~${v=2u}$ 
and~${\cL_uJ=0}$, and so~${\cL_vJ=0}$.
Conversely, if~${\cL_vJ=0}$, then~${\bL F+\bB G=2F}$ 
and~${\bL G-\bB F=2G}$ by part~(iii) of Lemma~\ref{le:LB},
hence the unique solution of~\eqref{eq:LBFGPP} is given by~${\Phi=\tfrac{1}{2}F}$
and~${\Psi=\tfrac{1}{2}G}$, which implies~${u=\tfrac{1}{2}v}$ and~${\cL_uJ=0}$,
so equality in~\eqref{eq:BERNDTSSON} follows from~\eqref{eq:uvFG}.
To prove the last assertion, 
define~${F_t:=F+t\Fhat}$ and~${G_t:=G+t\Ghat}$ and differentiate the 
function~${t\mapsto\int_M\bigl(\tfrac{1}{2}\abs{v_{F_t}+Jv_{G_t}}^2-F_t^2-G_t^2\bigr)\rho_J}$
at~${t=0}$. 
\end{proof}

\begin{proof}[Proof of Theorem~\ref{thm:DON}]
Fix an element~${J\in\sJ_\INT(M,\om)}$ 
and let~$\rho_J$ be as in~\eqref{eq:rhoJFANO}.
We show first that the $2$-form~\eqref{eq:DONSYMP} is nondegenerate
and compatible with the complex structure~${\Jhat\mapsto-J\Jhat}$.  
To see this, let~${\Jhat\in T_J\sJ_\INT(M,\om)}$,
let~${f,g}$ be as in~\eqref{eq:LAMBDAfgFANO}, 
and let~${F,G,A}$ be as in Lemma~\ref{le:DEC}. 
Then~\eqref{eq:DONSYMP} and~\eqref{eq:FGA} yield
\begin{equation}\label{eq:NONDEG1}
\begin{split}
\Om^D_J(\Jhat,-J\Jhat) 
&= 
\int_M\Bigl(\tfrac{1}{2}\trace\bigl(\Jhat^2\bigr) - f^2 - g^2\Bigr)\rho_J \\
&= 
\int_M\tfrac{1}{2}\trace\bigl(A^2\bigr)\rho_J
+
2\int_M\Bigl(
\Abs{v_F+Jv_G}^2 - 2\bigl(F^2+G^2\bigr)
\Bigr)\rho_J.
\end{split}
\end{equation}
By Lemma~\ref{le:BERNDTSSON} the right hand side in~\eqref{eq:NONDEG1} 
is nonnegative and vanishes if and only if~${A=0}$ 
and~${\cL_{v_F+Jv_G}J=0}$ or, equivalently,~${\Jhat=0}$. 
This proves nondegeneracy.

To prove~\eqref{eq:DONMOMENT2}, 
fix an element~${\Jhat\in T_J\sJ_\INT(M,\om)}$,
let~$f,g$ be as in~\eqref{eq:LAMBDAfgFANO},
and let~${H\in\Om^0(M)}$ such that~${\int_MH\rho_J=0}$.
Then, by Lemma~\ref{le:LB} and~\eqref{eq:LBFG}, we have
\begin{equation*}
\begin{split}
\Om^D_J(\Jhat,\cL_{v_H}J) 
&= 
\tfrac{1}{2}\int_M\trace\bigl(\Jhat J\cL_{v_H}J\bigr)\rho_J
- \int_M f\left(2H - \bL H\right)\rho_J 
+ \int_M g(\bB H)\rho_J  \\
&= 
\int_M\Lambda_{\rho_J}(J,\Jhat)\wedge\iota(v_H)\rho_J 
+ \int_M dH\wedge dg\wedge\Theta_J\frac{\om^{\sn-1}}{(\sn-1)!}  \\
&\quad
+ \int_M (df\circ J)\wedge dH\wedge\Theta_J\frac{\om^{\sn-1}}{(\sn-1)!}
- 2\int_MHf\rho_J \\
&=
-2\int_MHf\rho_J.
\end{split}
\end{equation*}
Here the last equality holds because~${\Lambda_{\rho_J}(J,\Jhat)=-df\circ J+dg}$.
This proves the first equality in~\eqref{eq:DONMOMENT2}
and the second follows from Lemma~\ref{le:THETAHAT}.

It remains to prove that the $2$-form~\eqref{eq:DONSYMP} is closed.
In Donaldson's formulation this follows directly from the definition,
while in our formulation this requires proof.  Here is the outline.
First, let~${\R^2\to\sJ_\INT(M,\om):(s,t)\mapsto J_{s,t}}$
be a smooth map and, for~${s,t\in\R}$, define the 
functions~${f_s,g_s,f_t,g_t\in\Om^0(M)}$ such that
they have mean value zero with respect to~${\rho_J}$
for~${J=J_{s,t}}$ and
$$
\Lambda_{\rho_J}(J,\p_sJ) = -df_s\circ J+dg_s,\qquad
\Lambda_{\rho_J}(J,\p_tJ) = -df_t\circ J+dg_t.
$$
Here we have dropped the subscripts~${s,t}$ for~$J$ and 
observe that~${f_s,g_s,f_t,g_t}$ depend also
on~$s$ and~$t$.   Then by~\cite[Theorem~2.7]{GPST}
and Lemma~\ref{le:THETAHAT} we have
\begin{equation*}
\begin{split}
\p_s\Lambda_{\rho_J}(J,\p_tJ)-\p_t\Lambda_{\rho_J}(J,\p_sJ)
+ \tfrac{1}{2}d\trace\bigl((\p_sJ)J(\p_tJ)\bigr) 
= 
df_s\circ \p_tJ - df_t\circ \p_sJ.
\end{split}
\end{equation*}
Hence a calculation shows that
\begin{equation}\label{eq:dsdtf} 
\begin{split}
\p_s f_t - \p_tf_s = 0,\qquad
d\Bigl(\p_s g_t - \p_tg_s
+ \tfrac{1}{2}\trace\bigl((\p_sJ)J(\p_tJ)\bigr) \Bigr) = 0.
\end{split}
\end{equation}
Now let
$
{\R^3\to\sJ_\INT(M,\om):(r,s,t)\mapsto J(r,s,t)}
$
be a smooth map and define the functions~${f_r,f_s,f_t,g_r,g_s,g_t}$
as before.  Then~${\p_r\rho_J=f_r\rho_J}$ by Lemma~\ref{le:THETAHAT} 
and hence
\begin{equation*}
\begin{split}
&\p_r\Om^D_J(\p_sJ,\p_tJ) 
=
\tfrac{1}{2}\int_M f_r\trace\bigl((\p_sJ)J(\p_tJ)\bigr)\rho_J \\
&\quad
+ \tfrac{1}{2}\int_M\trace\bigl((\p_r\p_sJ)J(\p_tJ)\bigr)\rho_J 
+ \tfrac{1}{2}\int_M\trace\bigl((\p_sJ)J(\p_r\p_tJ)\bigr)\rho_J \\
&\quad
+ \int_M\Bigl(
(\p_rg_s)f_t + g_s(\p_rf_t) + f_rg_sf_t
- (\p_rf_s)g_t - f_s(\p_rg_t) - f_rf_sg_t
\Bigr)\rho_J.
\end{split}
\end{equation*}
Take a cyclic sum and use~\eqref{eq:dsdtf} to 
obtain~${(d\Om^D)_J(\p_rJ,\p_sJ,\p_tJ)=0}$.
\end{proof}


\subsection{The Ding functional and the K\"ahler--Ricci flow}\label{subsec:DING}
Fix a complex structure~${J\in\sJ_\INT(M,\om)}$ and denote by~$\sH_J$ 
the space of K\"ahler potentials as in Remark~\ref{rmk:KPOT}.
The analogue of the Mabuchi functional in the present setting is the 
{\bf Ding functional}~${\cF_J:\sH_J\to\R}$, defined by 
\begin{equation}\label{eq:DING}
\cF_J(h) := \cI_J(h)-\log\left(\int_Me^h\rho_J\right) 
\end{equation}
for~${h\in\sH_J}$, where~${\cI_J:\sH_J\to\R}$ is 
the unique functional that satisfies
\begin{equation*}
\cI_J(0)=0,\qquad 
d\cI_J(h)\hhat= \frac{1}{V}\int_M\hhat\rho_h,\qquad
\rho_h := \frac{\om_h^\sn}{\sn!},
\end{equation*}
for all~${h\in\sH_J}$ and all~${\hhat\in\Om^0(M)}$.
An explicit formula is~${\cI_J(h):=\int_0^1\tfrac{1}{V}\int_Mh\rho_{th}\,dt}$.
For~${h\in\sH_J}$ define~${\theta_h:=\Theta_{\om_h,J}:M\to(0,\infty)}$
(see Definition~\ref{def:THETAJ}).  Then
\begin{equation}\label{eq:thetah}
\Ric_{\theta_h\rho_h,J} = \om_h,\qquad 
\int_M\theta_h\rho_h = 1,\qquad
\rho_h := \frac{\om_h^\sn}{\sn!}.
\end{equation}
Since~${\Ric_{\rho_J,J}=\om}$, 
we have~${\Ric_{e^h\rho_J,J}=\om_h=\Ric_{\theta_h\rho_h,J}}$ by~\eqref{eq:RICf},
and hence
\begin{equation}\label{eq:thetahJ}
\theta_h\rho_h=\frac{e^h\rho_J}{\int_Me^h\rho_J}
\end{equation}
for all~${h\in\sH_J}$.  This implies 
\begin{equation}\label{eq:dDING}
d\cF_J(h)\hhat 
= \frac{1}{V}\int_M\hhat\rho_h-\frac{\int_M\hhat e^h\rho_J}{\int_Me^h\rho_J}
= \int_M\hhat\left(\frac{1}{V}-\theta_h\right)\rho_h
\end{equation}
for~${h\in\sH_J}$ and~${\hhat\in\Om^0(M)}$. 
Thus the gradient of the Ding functional~$\cF_J$ with respect to the  
Riemannian metric~\eqref{eq:KPOTmetric} on~$\sH_J$ is given by
\begin{equation}\label{eq:gradDING}
\grad\cF_J(h) = \frac{1}{V}-\theta_h
\end{equation}
for~${h\in\sH_J}$. In~\cite{BB1} Berndtsson proved the following.

\begin{theorem}[{\bf Berndtsson}]\label{thm:BERNDTSSON}
The Ding functional is convex along geodesics.
\end{theorem}

\begin{proof}
Let~${I\to\sH_J:t\mapsto h_t}$ be a geodesic
so that~${\p_t\p_th+\tfrac{1}{2}\abs{d\p_th}_h^2=0}$.
Then it follows from~\eqref{eq:dDING} that
\begin{equation*}
\begin{split}
\frac{d^2}{dt^2}\cF_J(h)
&=
\frac{1}{V}\frac{d}{dt}\int_M(\p_th)\rho_h
- \frac{d}{dt}\frac{\int_M(\p_th)e^h\rho_J}{\int_Me^h\rho_J} \\
&=
- \frac{\int_M(\p_t\p_th)e^h\rho_J}{\int_Me^h\rho_J}
- \frac{\int_M(\p_th)^2e^h\rho_J}{\int_Me^h\rho_J}
+ \Abs{\frac{\int_M(\p_th)e^h\rho_J}{\int_Me^h\rho_J}}^2 \\
&=
\tfrac{1}{2}\frac{\int_M\Abs{d\p_th}_h^2e^h\rho_J}{\int_Me^h\rho_J}
- \frac{\int_M(\p_th)^2e^h\rho_J}{\int_Me^h\rho_J}
+ \Abs{\frac{\int_M(\p_th)e^h\rho_J}{\int_Me^h\rho_J}}^2 \\
&=
\tfrac{1}{2}\int_M\Abs{d\p_th}_h^2\theta_h\rho_h
- \int_M(\p_th)^2\theta_h\rho_h
+ \Abs{\int_M(\p_th)\theta_h\rho_h}^2 \\
&\ge
0.
\end{split}
\end{equation*}
Here the second equality holds 
because~${\int_M(\p_th)(\p_t\rho_h)=\tfrac{1}{2}\int_M\abs{d\p_th}_h^2\rho_h}$.
The last inequality holds by Lemma~\ref{le:BERNDTSSON} with~$\om$ replaced 
by~$\om_h$, with~$\rho_J$  replaced by~$\theta_h\rho_h$,
and with~${F:=\p_th-\int_M(\p_th)\theta_h\rho_h}$ and~${G:=0}$. 
\end{proof}

In finite-dimensional GIT the gradient flow of the moment map squared
translates into the gradient flow of the Kempf--Ness function.  
In the present setting the moment map is given 
by~${\sJ_\INT(M,\om)\to\Om^0_\rho(M):J\mapsto 2(1/V-\Theta_J)}$,
where~$\Theta_J$ is defined by~\eqref{eq:THETAJ}. 
It is convenient to take one eighth (instead of one half)
of the square of the moment map to obtain the energy 
functional~${\cE_\om:\sJ_\INT(M,\om)\to\R}$ defined by
\begin{equation}\label{eq:E}
\cE_\om(J) := \tfrac{1}{2}\int_M\Bigl(\frac{1}{V}-\Theta_J\Bigr)^2\frac{\om^\sn}{\sn!}
\end{equation}
for~${J\in\sJ_\INT(M,\om)}$.
Consider the Riemannian metric on~$\sJ_\INT(M,\om)$ 
determined by the symplectic form~\eqref{eq:DONSYMP} 
and the complex structure~${\Jhat\mapsto-J\Jhat}$.  It is given by
\begin{equation}\label{eq:DONmetric}
\begin{split}
\inner{\Jhat_1}{\Jhat_2}_J 
:= 
\Om^D_J(\Jhat_1,-J\Jhat_2) 
=
\int_M\Bigl(
\tfrac{1}{2}\trace(\Jhat_1\Jhat_2) - f_1f_2-g_1g_2
\Bigr)\rho_J
\end{split}
\end{equation}
for~${J\in\sJ_\INT(M,\om)}$ and~${\Jhat_i\in T_J\sJ_\INT(M,\om)}$,
where~${\rho_J,f_i,g_i}$ are as in~\eqref{eq:rhoJFANO}
and~\eqref{eq:LAMBDAfgFANO}.  By Lemma~\ref{le:THETAHAT} 
the differential of the functional~$\cE_\om$ in~\eqref{eq:E} is given by 
$$
d\cE_\om(J)\Jhat 
= \int_M f\Theta_J\rho_J
= - \tfrac{1}{2}\Om^D_J(\Jhat,\cL_vJ)
= \inner{\Jhat}{-\tfrac{1}{2}J\cL_vJ}_J,
$$
for~${J\in\sJ_\INT(M,\om)}$ and~${\Jhat\in T_J\sJ_\INT(M,\om)}$, 
where~$f$ is as in~\eqref{eq:LAMBDAfgFANO},
$v$ is the Hamiltonian vector field of~$\Theta_J$.
and the second equality follows from~\eqref{eq:DONMOMENT2}.
This shows that the gradient of~$\cE_\om$ at~$J$ with respect 
to the metric~\eqref{eq:DONmetric} is given by
\begin{equation}\label{eq:gradE}
\grad\cE_\om(J) = -\tfrac{1}{2}J\cL_vJ,\qquad
\iota(v)\om = d\Theta_J.
\end{equation}
Thus a complex structure~${J\in\sJ_\INT(M,\om)}$ is a critical point
of~$\cE_\om$ if and only if the Hamiltonian vector field of~$\Theta_J$ is holomorphic.  
Such a complex structure is called a {\bf Donaldson--K\"ahler--Ricci soliton}. 
By~\eqref{eq:gradE} a negative gradient flow line of~$\cE_\om$ 
is a solution~${I\to\sJ_\INT(M,\om):t\mapsto J_t}$ 
of the partial differential equation
\begin{equation}\label{eq:DONKRFLOWJ}
\p_tJ_t  =  \tfrac{1}{2}J_t\cL_{v_t}J_t,\qquad
\iota(v_t)\om = d\Theta_{J_t}.
\end{equation}
If~${t\mapsto J_t}$ is a solution of~\eqref{eq:DONKRFLOWJ}
on an interval~${I\subset\R}$ containing zero with~${J_0=J}$,
and~${I\to\Diff_0(M):t\mapsto\phi_t}$ is the isotopy defined 
by~${\p_t\phi_t+\tfrac{1}{2}J_tv_t\circ\phi_t=0}$, ${\phi_0=\id}$, 
then~${\phi_t^*J_t=J}$ for all~$t$ and the 
paths~${\om_t:=\phi_t^*\om}$ and~${\theta_t:=\Theta_{J_t}\circ\phi_t}$
satisfy
\begin{equation}\label{eq:DONKRFLOWom}
\p_t\om_t  =  \tfrac{1}{2}d(d\theta_t\circ J),\quad
\tfrac{1}{2}d(d\log(\theta_t)\circ J) = \om_t-\Ric_{\om_t^\sn/\sn!,J},\quad
\int_M\theta_t\frac{\om_t^\sn}{\sn!} = 1,
\end{equation}
This is the {\bf Donaldson--K\"ahler--Ricci flow}. 
Here~$J$ is a Fano complex structure and~\eqref{eq:DONKRFLOWom}
is understood as an equation for paths in the space~$\sS_J$ in~\eqref{eq:SJ} 
of all symplectic forms that are compatible with~$J$ 
and represent the cohomology class~${2\pi c_1^\R(J)}$. 
When~$\om\in\sS_J$ is fixed, a solution of~\eqref{eq:DONKRFLOWom} 
has the form~${\om_t=\om_{h_t}}$, where~${I\to\sH_J:t\mapsto h_t}$ 
is a smooth path satisfying
\begin{equation}\label{eq:DINGFLOW}
\p_th_t  =  \theta_{h_t} - \frac{1}{V}.
\end{equation}
By~\eqref{eq:gradDING} the solutions of~\eqref{eq:DINGFLOW}
are the negative gradient flow lines of the 
Ding functional~${\cF_J:\sH_J\to\R}$ in~\eqref{eq:DING}.  
The next remark shows that~\eqref{eq:DINGFLOW} 
is a second order parabolic partial differential equation.

\begin{remark}\label{rmk:DKR}
Let~$\nabla$ be the Levi-Civita connection of the 
metric~${\inner{\cdot}{\cdot}:=\om(\cdot,J\cdot)}$ 
and let~${h\in\sH_J}$. 
Then~${\rho_h=\det(\one-\tfrac{1}{2}\nabla^2h + \tfrac{1}{2}J(\nabla^2h)J)^{1/2}
\om^\sn/\sn!}$ and hence it follows from~\eqref{eq:thetahJ} that
$
\theta_h = (\int_Me^h\rho_J)^{-1}e^h\Theta_J
\det(\one-\tfrac{1}{2}\nabla^2h + \tfrac{1}{2}J(\nabla^2h)J)^{-1/2}.
$
\end{remark}

Define the functional~${\cH_\om:\sJ_\INT(M,\om)\to\R}$ by
\begin{equation}\label{eq:HE}
\cH_\om(J) := \int_M\log(V\Theta_J)\Theta_J\frac{\om^\sn}{\sn!}
\end{equation}
for~${J\in\sJ_\INT(M,\om)}$.  This functional was introduced by 
Weiyong He~\cite{HE} (as a functional on the space of K\"ahler 
potentials for a fixed complex structure).  It is nonnegative 
and vanishes on a complex structure~${J\in\sJ_\INT(M,\om)}$ if and only if
it satisfies the K\"ahler--Einstein condition~${\Ric_{\om^\sn/\sn!,J}=\om}$
(see part~(i) of Theorem~\ref{thm:HD} below).  
By Lemma~\ref{le:THETAHAT} the differential of the 
functional~$\cH_\om$ is given by
$$
d\cH_\om(J)\Jhat 
= \int_M f\log(\Theta_J)\rho_J
= \inner{\Jhat}{-\tfrac{1}{2}J\cL_vJ}_J
$$
for~${J\in\sJ_\INT(M,\om)}$ and~${\Jhat\in T_J\sJ_\INT(M,\om)}$,
where~$f$ is as in~\eqref{eq:LAMBDAfgFANO},
$v$ is the Hamiltonian vector field of~$\log(\Theta_J)$,
and the second equality follows from~\eqref{eq:DONMOMENT2}.
This shows that the gradient of~$\cH_\om$ at~$J$ with respect to the 
metric~\eqref{eq:DONmetric} is given by
\begin{equation}\label{eq:gradHE}
\grad\cH_\om(J)=-\tfrac{1}{2}J\cL_vJ,\qquad \iota(v)\om=d\log(\Theta_J).
\end{equation}
Thus a complex structure~${J\in\sJ_\INT(M,\om)}$ is 
a critical point of~$\cH_\om$ if and only if the Hamiltonian 
vector field of~$\log(\Theta_J)$ is holomorphic.  
Such a complex structure is called a {\bf K\"ahler--Ricci soliton}. 
By~\eqref{eq:gradHE} a negative gradient flow line of~$\cH_\om$
is a solution~${I\to\sJ_\INT(M,\om):t\mapsto J_t}$
of the partial differential equation
\begin{equation}\label{eq:KRFLOWJ}
\p_tJ_t  =  \tfrac{1}{2}J_t\cL_{v_t}J_t,\qquad
\iota(v_t)\om = d\log(\Theta_{J_t}),
\end{equation}
If~${t\mapsto J_t}$ is a solution of~\eqref{eq:KRFLOWJ}
on an interval~${I\subset\R}$ containing zero with~${J_0=J}$,
and~${I\to\Diff_0(M):t\mapsto\phi_t}$ is the isotopy defined 
by~${\p_t\phi_t+\tfrac{1}{2}J_tv_t\circ\phi_t=0}$, ${\phi_0=\id}$, 
then~${\phi_t^*J_t=J}$ for all~$t$ and the 
paths~${\om_t:=\phi_t^*\om}$ and~${\theta_t:=\Theta_{J_t}\circ\phi_t}$
satisfy the equation~${\p_t\om_t=\tfrac{1}{2}d(d\log(\theta_t)\circ J)}$.
With~${\rho_t:=\om_t^\sn/\sn!}$
we also have~${\theta_t\rho_t=\phi_t^*\rho_{J_t}}$,
hence~${\Ric_{\rho_t,J}+\tfrac{1}{2}d(d\log(\theta_t)\circ J)
=\Ric_{\theta_t\rho_t,J}=\phi_t^*\Ric_{\rho_{J_t},J_t}=\phi_t^*\om=\om_t}$,
and so 
\begin{equation}\label{eq:KRFLOWom}
\p_t\om_t  =  \om_t-\Ric_{\rho_t,J},\qquad
\rho_t:=\om_t^\sn/\sn!.
\end{equation}
This is the standard {\bf K\"ahler--Ricci flow} on the space~$\sS_J$
of all $J$-compatible symplectic forms in the class~$2\pi c_1^\R(J)$ 
associated to a Fano complex structure~$J$. 
When a symplectic form~$\om\in\sS_J$ is fixed, a solution 
of~\eqref{eq:KRFLOWom} has the form~${\om_t=\om_{h_t}}$, 
where~${I\to\sH_J:t\mapsto h_t}$ is a smooth path satisfying
\begin{equation}\label{eq:KRFLOWh}
\p_th_t=\log(V\theta_{h_t}).
\end{equation}
Now define the functional~${\cH_J:\sH_J\to\R}$ 
as in Weiyong He's original paper~\cite{HE} by 
\begin{equation}\label{eq:HEh}
\cH_J(h) := \cH_{\om_h}(J) = \int_M\log(V\theta_h)\theta_h\rho_h
\end{equation}
for~${h\in\sH_J}$, where~$\theta_h$ and~$\rho_h$ are as in~\eqref{eq:thetah}.
The properties of this functional with regard to the K\"ahler--Ricci
flow are summarized in the following theorem.
The first two assertions are due to He~\cite{HE}
and the last inequality is due to Donaldson~\cite{DON5}.

\begin{theorem}[{\bf He, Donaldson}]\label{thm:HD}
Fix a complex structure~${J\in\sJ_\INT(M,\om)}$.

\smallskip\noindent{\bf (i)}
Let~${h\in\sH_J}$.  Then~${\cH_J(h)\ge0}$ with equality
if and only if~${\Ric_{\rho_h,J}=\om_h}$.

\smallskip\noindent{\bf (ii)}
A K\"ahler potential~${h\in\sH_J}$ is a critical point of~$\cH_J$ 
if and only if it is a K\"ahler--Ricci soliton, 
i.e.\ the~$\om_h$-Hamiltonian vector field of~$\log(\theta_h)$ is holomorphic.

\smallskip\noindent{\bf (iii)}
Every solution~${I\to\sH_J:t\mapsto h_t}$ of~\eqref{eq:KRFLOWh}
satisfies the inequalities
\begin{equation}\label{eq:HD}
\frac{d}{dt}\cH_J(h_t)\le 0,\qquad
\frac{d}{dt}\cF_J(h_t)\le -\cH_J(h_t).
\end{equation}
\end{theorem}

\begin{proof}
Following~\cite{DON5},
we define the function~${B:(0,\infty)\to\R}$ by
\begin{equation}\label{eq:B}
B(x) := x\log(Vx) + \frac{1}{V}-x
\end{equation}
for~${x>0}$.  Then~${B'(x)=\log(Vx)}$ and~${B''(x)=1/x}$.
Hence~${B(1/V)=0}$ and~$B$ is strictly convex.
This implies~${B(x)>0}$ for every positive 
real number~${x\ne1/V}$ and
\begin{equation}\label{eq:Bprime}
\left(x-\frac{1}{V}\right)\log(Vx) 
= \left(x-\frac{1}{V}\right)B'(x)
\ge B(x) = x\log(Vx) + \frac{1}{V}-x.
\end{equation}
To prove part~(i), fix an element~${h\in\sH_J}$.  
Then, by~\eqref{eq:thetah}, \eqref{eq:HEh}, and~\eqref{eq:B},
$$
\cH_J(h) 
= \int_M\left(\theta_h\log(V\theta_h) + \frac{1}{V}-\theta_h\right)\rho_h
= \int_MB(\theta_h)\rho_h.
$$
Hence~${\cH_J(h)\ge0}$ with equality if and 
only if~${\theta_h=1/V}$.  This proves~(i). 

We prove part~(ii).   A calculation shows that
\begin{equation}\label{eq:gradHE1}
\begin{split}
d\cH_J(h)\hhat 
&= -\tfrac{1}{2}\int_M\inner{d\hhat}{d\log(V\theta_h)}_h\theta_h\rho_h 
+ \int_M \hhat\log(V\theta_h)\theta_h\rho_h  \\
&\quad\;
- \left(\int_M\log(V\theta_h)\theta_h\rho_h\right)\left(\int_M\hhat\theta_h\rho_h\right)
\end{split}
\end{equation}
for~${h\in\sH_J}$ and~${\hhat\in\Om^0(M)}$.
This implies
\begin{equation}\label{eq:gradHE2}
\begin{split}
d\cH_J(h)\log(V\theta_h) 
&= -\tfrac{1}{2}\int_M\abs{d\log(V\theta_h)}_h^2\theta_h\rho_h 
+ \int_M \bigl(\log(V\theta_h)\bigr)^2\theta_h\rho_h \\
&\quad\;
- \left(\int_M\log(V\theta_h)\theta_h\rho_h\right)^2 
\le 0
\end{split}
\end{equation}
for all~${h\in\sH_J}$.
Here the inequality follows from Lemma~\ref{le:BERNDTSSON}, 
with $\om,\rho_J$ replaced by $\om_h,\theta_h\rho_h$ 
and $F:=\log(V\theta_h)-\int_M\log(V\theta_h)\theta_h\rho_h$
and~${G:=0}$. It follows also from  Lemma~\ref{le:BERNDTSSON}
that~${d\cH_J(h)=0}$ if and only if~${d\cH_J(h)\log(V\theta_h)=0}$ 
if and only if the vector field~$v$ defined 
by~${\iota(v)\om_h=d\log(V\theta_h)}$ satisfies~${\cL_vJ=0}$.  
This proves~(ii). 

We prove part~(iii). The first inequality in~\eqref{eq:HD} 
follows directly form~\eqref{eq:gradHE2}. 
To prove the second inequality, recall from 
equation~\eqref{eq:dDING} that
\begin{equation*}
\begin{split}
d\cF_J(h)\log(V\theta_h)
&= 
\int_M\log(V\theta_h)\Bigl(\frac{1}{V}-\theta_h\Bigr)\rho_h \\
&\le 
- \int_M\Bigl(\theta_h\log(V\theta_h) + \frac{1}{V}-\theta_h\Bigr)\rho_h 
=
- \cH_J(h).
\end{split}
\end{equation*}
Here the second step follows from~\eqref{eq:Bprime}.
This proves~(iii) and the theorem.
\end{proof} 

In~\cite{DON5} Donaldson noted the following.
If~${[0,\infty)\to\sH_J:t\mapsto h_t}$ is a solution 
of the K\"ahler--Ricci flow~\eqref{eq:KRFLOWh}
and the limit~${h:=\lim_{t\to\infty}h_t}$ exists in~$\sH_J$, 
but the pair~$(\om_h,J)$ is not a K\"ahler--Einstein structure,
then it follows from Theorem~\ref{thm:HD} 
that~${\cH_J(h_t)\ge\cH_J(h)>0}$ and hence the 
Ding functional~${\cF_J(h_t)}$ diverges 
to minus infinity as~$t$ tends to infinity.
This corresponds to the observation in GIT
that the Kempf--Ness function of an unstable point is unbounded below.  
The analogue of the Kempf--Ness Theorem in the present setting 
would be the assertion
\begin{equation}\label{eq:UNSTABLEconjecture}
\inf_{h\in\sH_J}\int_M\Bigl(\frac{1}{V}-\theta_h\Bigr)^2\rho_h > 0
\qquad\iff\qquad
\inf_{h\in\sH_J}\cF_J(h) = -\infty
\end{equation}
for every~${J\in\sJ_\INT(M,\om)}$.  This seems to be an open question.


\end{document}